\documentclass[a4paper,11pt]{amsart}

\usepackage[all, cmtip]{xy}
\usepackage{amsfonts}
\usepackage{amsmath}
\usepackage{amssymb}
\usepackage{amsthm}
\usepackage{enumerate}
\usepackage{hyperref}
\usepackage{geometry}
\usepackage{mathabx}

\title{Derived Algebraic Cobordism}
\author{Parker E. Lowrey}
\address{Department of Mathematics \\Middlesex College \\The University 
  of Western Ontario \\London, On \\Canada}
\email{plowrey@uwo.ca}

\author{Timo Sch\"urg}
\address{Mathematisches Institut \\Universit\"at Bonn \\Endenicher Allee 
  60\\53115 Bonn\\Germany}
\email{timo\_schuerg@operamail.com}

\DeclareMathOperator{\cl}{cl}

\DeclareMathOperator{\Sup}{Sup}

\DeclareMathOperator{\QCoh}{QCoh}
\DeclareMathOperator{\virdim}{virdim}

\DeclareMathOperator{\Spec}{Spec}
\DeclareMathOperator{\pt}{pt}

\DeclareMathOperator{\Tor}{Tor}

\DeclareMathOperator{\id}{id}

\DeclareMathOperator{\pr}{pr}
\DeclareMathOperator{\Sym}{Sym}

\DeclareMathOperator{\Pre}{pre}

\DeclareMathOperator{\naive}{naive}
\DeclareMathOperator{\FGL}{FGL}
\DeclareMathOperator{\SNC}{SNC}

\newcommand{\Cat}[1]{\mathbf{#1}}

\newcommand{\IA}{\mathbb{A}}

\newcommand{\IL}{\mathbb{L}}

\newcommand{\IP}{\mathbb{P}}

\newcommand{\IR}{\mathbb{R}}

\newcommand{\IZ}{\mathbb{Z}}

\newcommand{\sR}{\mathcal{R}}

\newcommand{\sM}{\mathcal{M}}
\newcommand{\sO}{\mathcal{O}}

\newtheorem{thm}{Theorem}[section]
\newtheorem{lem}[thm]{Lemma}
\newtheorem{prop}[thm]{Proposition}
\newtheorem{cor}[thm]{Corollary}

\theoremstyle{definition}
\newtheorem{defn}[thm]{Definition}
\newtheorem{exmp}[thm]{Example}

\theoremstyle{remark}
\newtheorem{rem}[thm]{Remark}

\def\signed #1{{\leavevmode\unskip\nobreak\hfil\penalty50\hskip2em
    \hbox{}\nobreak\hfil #1%
  \parfillskip=0pt \finalhyphendemerits=0 \endgraf}}

\newsavebox\mybox

\begin{document}
\begin{abstract}
  We construct a cohomology theory using quasi-smooth derived schemes as 
  generators and an analogue of the bordism relation using derived fibre 
  products as relations. This theory has pull-backs along all morphisms 
  between smooth schemes independent of any characteristic assumptions.  
  We prove that in characteristic zero, the resulting theory agrees with 
  algebraic cobordism as defined by Levine and Morel. We thus obtain a 
  new set of generators and relations for algebraic cobordism.
\end{abstract}

\maketitle

\section{Introduction}

In his treatise on the universality of the formal group law of complex 
oriented cobordism \cite{quillen}, Quillen introduced a geometric set of 
generators and relations for complex oriented cobordism. He viewed 
complex cobordism as the universal contravariant functor on smooth 
manifolds endowed with Gysin homomorphisms for proper oriented maps. The 
distinctive feature of complex cobordism is that it comes equipped with 
a first Chern class operator for complex line bundles, and this Chern 
class operator satisfies the universal formal group law.  Quillen's 
construction of this theory is surprisingly simple. In his geometric 
description, cobordism classes in the cobordism group $U^q(X)$ can be 
represented by proper complex oriented maps $Z \to X$, and two such maps 
give the same cobordism class if they arise as fibres of a proper 
complex oriented map $W \to X \times \IR$. The construction of 
pull-backs in this geometric description for all maps between smooth 
manifolds relies on Thom's tranversality theorem. Since any two maps 
between manifolds can be moved by a homotopy until they are transversal, 
the pull-back of any generator $Z \to X$ will again be a proper complex 
oriented map between smooth manifolds.

A corresponding algebraic theory was introduced by Levine and Morel 
\cite{lm}. Compared to complex cobordism, it has a distinctly different 
flavor in two ways. First of all, the formal group law controlling the 
first Chern class operator does not appear in a natural way. In the 
algebraic setting, the formal group law corresponds to a rule on how to 
take a divisor consisting of several components apart. Surprisingly 
enough, as shown by Levine and Pandharipande in \cite{lp}, imposing a 
rule on the simplest possible configuration of divisors is already 
sufficient to obtain a formal group law for the first Chern class 
operator.  Alternatively, the formal group law can be imposed formally 
as done in \cite{lm}. On this issue derived algebraic geometry does not 
offer any insight. The second difference to the topological theory 
arises in the construction of pull-backs along all morphisms between 
smooth schemes.  Since algebraic geometry is by nature much more rigid, 
two morphisms between smooth schemes can in general not be arranged to 
be transversal.  Taking smooth schemes as representatives of algebraic 
cobordism classes, it therefore is difficult to obtain pull-backs in 
algebraic cobordism.  In fact, this is by far the most difficult part in 
the construction of \cite{lm} and is only carried out in characteristic 
zero. After these difficulties are out of the way, the resulting theory 
has a close resemblance to Quillen's theory of complex oriented 
cobordism. For instance, algebraic cobordism satisifies a similar 
universal property, and the cobordism ring of a point is isomorphic to 
the Lazard ring.

A natural idea how to obtain a cohomology theory having pull-backs along 
all morphisms between smooth schemes is to enlarge the class of schemes 
allowed as generators of the theory. Staying in the realm of algebraic 
geometry, it is unclear though which class of schemes to pick. The least 
complicated schemes after smooth schemes are local complete 
intersections, and these are also not stable under pull-backs between 
morphisms of smooth schemes. Here derived algebraic geometry offers a 
natural solution. The class of quasi-smooth derived schemes is a mild 
generalization of smooth schemes and local complete intersections. It 
satisfies excellent stability properties under all derived pull-backs of 
morphisms between smooth schemes, so that any cohomology theory defined 
with quasi-smooth derived schemes as generators will automatically have 
pull-backs along such morphisms. The class of quasi-smooth schemes 
offers a natural algebraic-geometric analogue of moving by a homotopy to 
transversal intersection, see e.g. the work of Ciocan-Fontanine and 
Kapranov in \cite{ck}. Quasi-smooth schemes are also not too far a 
generalization from schemes. For instance, the simplicial commutative 
rings giving the local theory of quasi-smooth schemes only have finitely 
many homotopy groups.  Furthermore, locally every quasi-smooth derived 
scheme can be written as the derived fibre product of un-derived 
schemes.

The aim of this paper is to study the theory obtained by using an 
algebraic version of Quillen's construction with quasi-smooth derived 
schemes replacing smooth manifolds as representatives of bordism 
classes, and by replacing deformation to transversal intersection by 
derived fibre products in the definition of the bordism relation. The 
resulting theory is called \emph{derived algebraic cobordism}. By 
construction, the resulting cohomology theory will have pull-backs along 
all morphisms between smooth manifolds, independent of any 
characteristic assumption on the base field. In particular, we obtain a 
first Chern class operator for line bundles. As in the algebraic theory 
of Levine and Morel, it is too much to expect that this operator will 
satisfy a formal group law.  After imposing this formal group law, we 
obtain a theory that resembles the theory of Levine and Morel. In 
characteristic zero we prove that there is a natural comparison map 
between algebraic bordism and derived algebraic bordism, and this map is 
in fact an isomorphism. We thus obtain a new set of generators and 
relations for algebraic cobordism.

The motivation came from finding an analog of Spivak's \cite{spivak} and 
Joyce's \cite{joyce} results in derived differential geometry and 
d-manifolds, respectively.  Using generators and relations as prescribed 
above, Joyce and Spivak define derived cobordism for any smooth 
manifold, and both are able to prove that derived cobordism agrees with 
classical cobordism.  Our result is slightly stronger in that it shows 
for any quasi-projective derived scheme (including both singular and 
derived), the aforementioned isomorphism exists. 

This result also provides a conceptual explanation why the virtual 
fundamental classes of Behrend--Fantechi \cite{bf} and Li--Tian 
\cite{lt} exist. It implies that every quasi-smooth derived scheme of 
virtual dimension $d$ is in fact bordant to a smooth scheme of dimension 
$d$. 

The method of proof employed here is remarkably similar to the proofs of 
Joyce and Spivak. For those acquainted with the proof, we have the 
following rough analogies.
\begin{table}[!h]
  \begin{tabular}{r|l}
    differential geometric& algebro geometric \\
    \hline
    manifold & smooth scheme\\
    derived manifold & quasi-smooth derived scheme\\
    tubular neighborhood & deformation to the normal cone\\
    Thom's transversality & Levine's moving lemma
  \end{tabular}
\end{table}

More precisely, following ideas from \cite{manolache}, we introduce 
orientations for quasi-smooth morphisms in algebraic bordism. In the 
case of a quasi-smooth derived scheme mapping to a point, this gives 
exactly the virtual fundamental class. Using the tools summarized in the 
above table, we then prove a Grothendieck--Riemann--Roch result stating 
that these orientations are compatible with the orientations defined for 
derived algebraic bordism (Corollary \ref{cor:GRR}). As a consequence, 
we obtain that the natural transformation
\[
  \vartheta_{\Omega} \colon d\Omega_* \longrightarrow \Omega_*
\]
obtained from the universal property of $d\Omega_*$ is an isomorphism.

An immediate consequence of this Grothendieck--Riemann--Roch result is 
that derived and classical bordism coincide as homology theories on the 
category of derived quasi-projective schemes.

\medskip
\noindent {\bf Theorem \ref{thm:dOmEqOm}}\quad
  {\sl For all $X \in \Cat{dQPr}_k$ the morphism
  \[
    \vartheta_{d\Omega} \colon \Omega_{*} (X)
    \longrightarrow d\Omega_{*}(X)
  \]
  is an isomorphism.}
\medskip

To emphasize the analogy between derived fibre products and deformation 
to transversal intersection, we prove the following formula for the 
intersection product in derived algebraic cobordism:

\medskip
\noindent {\bf Theorem \ref{thm:IntProd}}\quad
{\sl  Let $X \in \Cat{Sm}_k$. Then the intersection product is 
  represented by the homotopy fibre product:
    \[
          [Y] \cdot [Z] = [Y \times^h_X Z] \in d\Omega^*(X).
            \]

  }
\medskip

We briefly review the argument and the contents of the paper.

In Section \ref{sect:obmf} we introduce the abstract notion of an 
oriented Borel--Moore functor of geometric type with quasi-smooth 
pull-backs. We then define derived algebraic bordism in Section 
\ref{sect:dOm} and show that it is the universal oriented Borel--Moore 
functor of geometric type with quasi-smooth pull-backs. In Section 
\ref{sect:qSmPb} we study some properties of the underived bordism 
$\Omega_*$. We first extend $\Omega_*$ to a homology theory on derived 
quasi-projective schemes and verify that is the universal oriented 
Borel--Moore homology theory of geometric type. This yields a natural 
transformation $\vartheta_{d\Omega} \colon \Omega_* \to d\Omega_*$. The 
remainder of the section is devoted to construncting pull-backs along 
quasi-smooth morphisms in $\Omega_*$. Once this is done, we obtain, 
using the universal property of $d\Omega_*$, a natural transformation 
$\vartheta_{\Omega_*} \colon d\Omega_* \to \Omega_*$ in the opposite 
direction. Finally, in \ref{sect:Spivak} we show that these natural 
transformations are inverse to each other using a 
Grothendieck--Riemann--Roch type theorem.

The authors would like to thank Barbara Fantechi for explaining the 
virtual pull-backs of \cite{manolache} and Gabriele Vezzosi for the 
suggestion of imposing the formal group law instead of trying to prove 
it exists, which the authors tried unsuccessfully for quite some time.  
We would also like to thank David Ben-Zvi for helpful comments on 
various drafts of this work, and Bertrand To\"en for some help with 
deformation to the normal cone for derived schemes.

Finally we would like to thank the Mathematical Institute of the 
University of Bonn for its hospitality, and the SFB/TR 45 `Periods, Moduli Spaces and 
Arithmetic of Algebraic Varieties' of the DFG (German Research 
Foundation) for financial support.

\subsection*{Notation}
We will work throughout over a base field $k$. Starting from Section
\ref{sect:qSmPb}, this field will be of characteristic zero.

For any scheme or derived scheme $X$ we will denote by $\QCoh(X)$ the 
$\infty$-category of quasi-coherent sheaves on $X$ introduced by Lurie 
in \cite{DAGVIII}.  Roughly, objects of $\QCoh(X)$ correspond to 
possibly unbounded complexes of quasi-coherent sheaves on $X$.

We have adopted using $\mathbb{L}$ for the Lazard ring and $L_X$ for the 
cotangent complex of a scheme $X$.

In the text, various categories of schemes and derived schemes are 
encountered. We have used bold font with hopefully self-explanatory 
names for theses categories. For instance, $\Cat{dQPr}_k$ denotes the 
category of derived quasi-projective schemes over $k$, $\Cat{Sm}_k$ 
denotes the category of smooth quasi-projective schemes over $k$ and so 
forth.

Given a scheme $X$ and the data of an effective Cartier divisor on $X$, 
$D$, we do not distinguish between the Cartier divisor and the natural 
sub-scheme it generates.  We let $\sO_X(D)$ be the associated line 
bundle on $X$, with the natural section implicit.  We let $|D|$ denote 
the support of the Cartier divisor.   This is the reduced sub-scheme of 
$D$.

Similar to Fulton's convention \cite[Convention 1.4]{fulton}, if 
$i\colon Y \hookrightarrow X$ is a closed embedding and $\alpha \in 
d\Omega_*(Y)$, when no confusion can arise we write $\alpha \in 
d\Omega_*(X)$ rather then $i_\ast \alpha$.

Throughout this text, $f\colon X \to Y$ is transverse to $g\colon Z \to 
Y$ if they are Tor-independent and $X\times_Y Z \to Z$ is smooth. Here 
Tor-independence means that $\Tor_i^{\sO_Y}(\sO_X,\sO_Z)=0$ for all $i 
>0$.

\section{Oriented Borel--Moore functors on derived schemes}
\label{sect:obmf}

In this section we introduce oriented Borel--Moore functors on 
quasi-projective derived schemes which have quasi-smooth pullbacks. We 
believe this is the right setting for studying virtual fundamental 
classes, since for any quasi-projective quasi-smooth derived scheme $X$ 
we can then define the virtual fundamental class via $\pi_X ^{*} [1]$, 
where $\pi_X \colon X \to \pt$ is the structure morphism.

Following Levine--Morel \cite{lm}, we first introduce oriented 
Borel--Moore functors with product. In these theories one has no control 
over the behavior of the first Chern class operator. We later pass to 
oriented Borel--Moore functors of geometric type. There the first Chern 
class satisfies a formal group law.

\subsection{Oriented Borel--Moore functors with product}

The definition of an oriented Borel--Moore functor with product on 
derived schemes is an immediate generalization of the original 
definition of Levine and Morel in \cite{lm} for schemes. Let 
$\Cat{dQPr}_k$ denote the category of quasi-projective derived schemes 
over $k$. Let $\Cat{dQPr}'_k$ denote the subcategory of $\Cat{dQPr}_k$ 
with proper morphisms.
\begin{defn}
  \label{defn:BMF}
  An oriented Borel--Moore functor with product on $\Cat{dQPr}_k$ is 
  given by:
  \begin{enumerate}[({D}1)]
    \item An additive functor $A_* \colon \Cat{dQPr}'_k \to 
      \Cat{Ab}_{*}$.
    \item For each smooth morphism $f \colon X \to Y$ in $\Cat{dQPr}_k$ 
      of pure relative dimension $d$, a homomorphism of graded abelian 
      groups
      \[
        f^* \colon A_*(Y) \longrightarrow A_{*+d}(X)
      \]
    \item For each line bundle $L$ on $X$, a homomorphism of graded 
      abelian groups
      \[
        c_1 (L) \colon A_*(X) \longrightarrow A_{*-1}(X).
      \]
    \item For each pair $(X,Y)$ in $\Cat{dQPr}_k$, a bilinear graded 
      pairing
      \[
        \times \colon A_*(X) \times A_*(Y) \longrightarrow A(X \times Y)
      \]
      which is commutative, associative, and admits a distinguished 
      element $1 \in A_0(\pt)$ as a unit.
  \end{enumerate}
  These data are required to satisfy the following axioms:
  \begin{enumerate}[({A}1)]
    \item Let $f \colon X \to Y$ and $g \colon Y \to Z$ be smooth 
      morphisms in $\Cat{dQPr}_k$ of pure relative dimension. Then
      \[
        (g \circ f)^* = f^* \circ g^*.
      \]
      Moreover, $\id_X^* = \id_{A_*(X)}$.
    \item Let $f \colon X \to Z$ and $g \colon Y \to Z$ be morphisms in 
      $\Cat{dQPr}_k$, where $f$ is proper and $g$ is smooth of pure 
      relative dimension. Let
      \[
        \xymatrix{
          W \ar[r]^{g'} \ar[d]_{f'} & X \ar[d]^{f} \\
          Y \ar[r]_{g} & Z
        }
      \]
      be the resulting Cartesian square. Then
      \[
        g^*f_* = f'_*g^{\prime *}.
      \]
    \item Let $f \colon X \to Y$ be proper and let $L \to Y$ be a line 
      bundle. Then
      \[
        f_* \circ c_1(f^* L) = c_1(L) \circ f_*.
      \]
    \item Let $f \colon X \to Y$ be smooth of pure relative dimension, 
      and let $L \to Y$ be a line bundle. Then
      \[
        c_1(f^*L) \circ f^* = f^* \circ c_1(L).
      \]
    \item For all line bundles $L$ and $M$ on $X \in \Cat{dQPr}_k$ we 
      have
      \[
        c_1(L) \circ c_1 (M) = c_1(M) \circ c_1 (L).
      \]
      If $L$ and $M$ are isomorphic as line bundles on $X$, then 
      $c_1(L)=c_1(M)$ holds.
    \item For proper morphisms $f$ and $g$ we have
      \[
        \times \circ (f_* \times g_*) = (f \times g)_* \circ \times.
      \]
    \item For smooth morphisms $f$ and $g$ we have
      \[
        \times \circ (f^* \times g^*) = (f \times g)^* \circ \times.
      \]
    \item For $X,Y \in \Cat{dQPr}_k$ and $L \to X$ a line bundle we have
      \[
        (c_1(L)(\alpha)) \times \beta = c_1(p_1^*(L))(\alpha \times 
        \beta).
      \]
  \end{enumerate}
\end{defn}

\begin{rem}
  It might seem surprising that in (A2) the square is only required to 
  be Cartesian and not homotopy Cartesian. But in this case the 
  Cartesian product is a homotopy fiber product since $g$ is assumed to 
  be smooth and thus flat.
\end{rem}

The following Lemma follows immediately from the definitions.

\begin{lem}
  \label{lem:restriction}
  An oriented Borel--Moore functor with product on $\Cat{dQPr}_k$ 
  restricted to $\Cat{QPr}_k$ is an Borel--Moore functor with product on 
  $\Cat{QPr}_k$.
\end{lem}

For a general oriented Borel--Moore functor one has no control over the 
behavior of the first Chern class operator. The next better behaved 
type of homology theory introduced by Morel and Levine is an oriented 
Borel--Moore functor of geometric type. For such a theory the first Chern 
class operator is controlled by a formal group law. Since the conditions 
only involve smooth schemes and a smooth derived scheme is always 
classical, the definition carries over immediately to derived schemes.  
We briefly recall the notion for the reader's convenience.

Recall that an \emph{oriented Borel--Moore $R_*$-functor with product} is 
an oriented Borel--Moore functor $A_*$ equipped with a graded ring 
homomorphism $R_* \to A_*(k)$. Here $R_*$ is graded unital commutative 
ring.

\begin{defn}
  \label{defn:BMFgt}
  Let $\IL_*$ be the Lazard ring, graded such that the universal formal 
  group law has total degree -1. An \emph{oriented Borel--Moore functor 
    of geometric type} is an oriented Borel--Moore $\IL_*$-functor $A_*$ 
  on $\Cat{dQPr} _k$ satisfying the following additional axioms:
  \begin{enumerate}
    \item[(Dim)] For a smooth scheme $X$ and line bundles $L_1, \, \dots 
      , L_r$ on $X$ with $r > \dim (X)$ we have
      \[
        c_1(L_1) \circ \dots \circ c_1 (L_r) (1_X) =0 \in A_*(X).
      \]
    \item[(Sect)] For a smooth scheme $X$ and a section $s \colon X \to 
      L$ of a line bundle $L$ on $X$ transverse to the zero section, we 
      have
      \[
        c_1 (L) (1_X) = i_* (1_Z),
      \]
      where $Z$ is the zero-set of $s$ and $i \colon Z \to X$ is the 
      inclusion.
    \item[(FGL)] For a smooth scheme $X$ and line bundles $L,M$ on $X$, 
      we have
      \[
        F_A (c_1(L),c_1(M))(1_X) = c_1 (L \otimes M)(1_X) \in A_*(X),
      \]
      where $F_A$ is the image of the universal formal group law on 
      $\IL$ under the homomorphism given by the $\IL_*$-structure.
  \end{enumerate}
\end{defn}

\begin{rem}
  \label{rem:loc}
  In \cite{lm}, Borel--Moore functors satisfying more axioms than those 
  of a functor of geometric type are considered. One of the axioms for a 
  weak homology theory is a weak localization axiom closely related to 
  the axiom (Sect) of a functor of geometric type:
  \begin{enumerate}
    \item [(Loc)] Let $L$ be a line bundle on a scheme $X$ admitting a 
      section $s \colon X \to L$ which is transverse to zero. Let $i 
      \colon Z \to X$ be the inclusion of the zero-set of $s$. Then the 
      image of $c_1(L) \colon A_*(X) \to A_{*-1}(X)$ is contained in the 
      image of $i_* \colon A_{*-1}(Z) \to A_{*-1}(X)$.
  \end{enumerate}
\end{rem}

Again, the following lemma is immediate from the definitions.
\begin{lem}
  \label{lem:restrictionGT}
  An oriented Borel--Moore functor of geometric type on $\Cat{dQPr}_k$ 
  restricted to $\Cat{QPr}_k$ is an Borel--Moore functor of geometric 
  type on $\Cat{QPr}_k$.
\end{lem}

If, in addition to the axiom (Loc) of Remark \ref{rem:loc}, one requires 
the Projective Bundle Theorem, and an extended homotopy relation, then 
one obtains the definition of a \emph{weak homology theory}. From there, 
one can further ask for pullback along locally complete intersection 
morphisms and certain cellular decomposition formulas to obtain a well 
behaved functor called a \emph{Borel--Moore homology theory}.

\subsection{Borel--Moore functors with quasi-smooth pull-backs}

To introduce virtual fundamental classes in an oriented 
Borel--Moore functor it is necessary to define orientations, i.e., 
pullbacks,  for quasi-smooth morphisms. We first recall some basic 
notions on quasi-smooth morphisms.

\begin{defn}
  A morphism $f \colon X \to Y$ of derived schemes is 
  \emph{quasi-smooth} if $f$ is locally of finite presentation and the 
  relative cotangent complex $L_{X/Y}$ is of Tor-amplitude $\leq 1$.
\end{defn}

\begin{exmp}
  Let $f \colon X \to Y$ be a local complete intersection morphism of 
  schemes. Then $f$ is quasi-smooth.
\end{exmp}

\begin{rem}
  Note that being locally of finite presentation as a morphism of 
  derived schemes is in general a stronger condition than requiring the 
  underlying morphism of schemes to be locally of finite presentation in 
  the usual sense. For instance, being locally of finite presentation as 
  morphism of derived schemes implies that the relative cotangent 
  complex is perfect. This is only true for local complete intersection 
  morphisms of classical schemes.
\end{rem}

For any point $p \colon \Spec A \to X$ the above condition implies that 
$p^* L_{X/Y}$ is locally isomorphic to a two-term complex of vector 
bundles. This leads to the definition of virtual dimension.

\begin{defn}
  \label{defn:virdim}
  Let $f \colon X \to Y$ be a quasi-smooth morphism, and let $p \colon 
  \Spec k \to X$ be a $k$-point of $X$.  Then the \emph{virtual 
    dimension of $f$ at $p$} is defined as
  \[
    \virdim (f,p) = H_{0}(p^* L_{X/Y}) - H_{1}(p^* L_{X/Y}).
  \]
\end{defn}

\begin{rem}
  The virtual dimension of $f$ at $p \in X$ is a locally constant 
  function.
\end{rem}

For later uses we will introduce quasi-smooth morphisms that admit a 
factorization into a quasi-smooth embedding followed by a smooth 
morphism. In the case of local complete intersection morphisms of 
schemes this is often built into the definition. As a rule of thumb, in 
English texts a local complete intersection has a global factorization 
by definition, e.g., \cite{fulton, lm}, whereas in French texts this is 
an additional property, e.g., \cite{verdier}.

\begin{defn}
  \label{defn:smoothable}
  We say that a morphism $f \colon X \to Y$ of derived schemes is 
  \emph{smoothable} if it factors as closed embedding followed by smooth 
  morphism. Such a factorization is called a \emph{smoothing}.
\end{defn}

\begin{rem}
  Let $f \colon X \to Y$ be a quasi-smooth morphism, and let $X 
  \overset{i}{\hookrightarrow} M \overset{p}{\longrightarrow} Y$ be a 
  smoothing. It then follows that $i \colon X \hookrightarrow M$ 
  is a quasi-smooth embedding.
\end{rem}

\begin{rem}
  Locally a smoothing exists for any quasi-smooth morphism.
\end{rem}

\begin{exmp}
  Any morphism $f \colon X \to Y$ with $X, Y \in \Cat{dQPr}_k$ admits a 
  smoothing.
\end{exmp}

Once we have pull-backs along quasi-smooth morphism we will, in 
particular, have fundamental classes for local complete intersections. We need
a normalization property for these fundamental 
classes to ensure compatibility with fundamental classes for local 
complete intersections in other homology theories. This normalization 
is automatically satisfied for Borel--Moore homology theories.

To state the desired normalization property, we first borrow notation from \cite[\S3.1]{lm}.  Given a formal group law $F(u,v) \in R[[u,v]]$ for some 
commutative ring $R$ there exists the \emph{difference group law}
\[
  F^{-}(u,v) \in R[[u,v]].
\]
If we let $\chi(u)$ denote the unique inverse power series satisfying $F(u, 
\chi (u))=0$, this difference group law is defined by the equation
\[
  F(u, v) = F^-(u,\chi(v)).
\]
Often the suggestive notation $+_F$ is used in place of the formal 
group law and $-_F$ for the difference. Given an integer $n$, denote
\[ [n]_F \cdot u := \begin{cases}  u +_F \ldots +_F u & n \geq 0 \\ u -_F \ldots -_F u & n < 0\end{cases}
\] 
where the operation is performed $|n|$ times, and given integers $n_1, \ldots, n_m$   
\[ F^{n_1, \ldots, n_m}(u_1, \ldots, u_m) := [n_1]_F \cdot u_1 +_F [n_2]_F \cdot u_2 \ldots +_F [n_m]_F\cdot u_m 
\]
One can guess many 
identities among the formal power series with this notation. For instance
\[
  (u +_F v)-(0+_Fv)=u
\]
translates to the formal power series identity
\[
  F^{-}(F(u,v),F(0,v))=u
\]
used in Lemma \ref{lem:c1L}.

Let $E$ be a strict normal crossing divisor on a smooth scheme $X$ with 
support $|E|$.  Following \cite{lm}, if $A_\ast$ is any Borel--Moore functor with first Chern classes 
obeying a formal group law and having proper push-forwards (this included Borel--Moore functors of geometric type), there exists a class $[E \to 
|E|] \in A_*(|E|)$ defined as follows.   Writing $E = \sum_{j=1}^m n_j E_j$ with each 
$E_j$ integral, for any index $J=(j_1, \, \dots , \, j_m)$ 
with $||J|| \leq 1$\footnote{Here $||J||:= \Sup_i(j_i)$} let  
$E^J:= \cap _{i, j_i=1} E_i$ be the ``$J$-th face'' and $i^J: E^J \to |E|$ the natural inclusion.   If $L_i = \sO_X(E_i)$ and $L_i^J = (i^J)^* L_i$, define
\begin{equation}
  \label{eq:FundSNC}
  [E \to |E|] = \sum_{J, ||J||\leq 1} i_*^J \left( [F_J^{n_1,\, \dots, 
      \, n_m}(L_1^J, \, \dots , \, L_m^J)]\right).
\end{equation}
With the notation now set, we define the central notion of this section.

\begin{defn}
  \label{defn:BMFqs}
  A \emph{oriented Borel--Moore functor with quasi-smooth pullbacks} (also referred to as ``with quasi-smooth orientations'', or ``with quasi-smooth Gysin-maps'') on 
  $\Cat{dQPr}_k$ consists of an oriented Borel--Moore functor $A_{*}$ of 
  geometric type equipped with:
  \begin{enumerate}
    \item[(B1)] For each equi-dimensional quasi-smooth morphism $f 
      \colon Y \to X$ of relative virtual dimension $d$ a homomorphism 
      of graded abelian groups
      \[
        f^{*} \colon A_{*}(X) \longrightarrow A_{*+d}(Y).
      \]
  \end{enumerate}
  These pull-backs should satisfy the following axiom
  \begin{enumerate}[({QS}1)]
    \item Let $s \colon L \to X$ be a section of a line bundle. Then 
      $c_1(L)=s^*s_*$.
    \item Let $f \colon X \to Y$ and $g \colon Y \to Z$ be quasi-smooth 
      morphisms of pure relative virtual dimension. Then
      \[
        (g \circ f)^* = f^*g^*.
      \]
    \item Let $f \colon X \to Z$ and $g \colon Y \to Z$ be morphisms 
      giving the homotopy Cartesian square
      \[
        \xymatrix{
          W \ar[r]^{g'} \ar[d]_{f'} & X \ar[d]^{f} \\
          Y \ar[r]_{g} & Z
        }
      \]
      with $f$ proper and $g$ quasi-smooth and virtually 
      equi-dimensional.
      Then
      \[
        g^{*}f_{*} = f'_{*} g^{\prime *}
      \]
    \item For quasi-smooth morphisms $f$ and $g$ we have
      \[
        \times \circ (f^* \times g^*) = (f \times g)^* \circ \times.
      \]
    \item For any strict normal crossing divisor $E$ in a smooth scheme 
      $X$ we have
      \[
        \pi_E^*([1])= (\zeta_{E})_\ast ([E \to |E|]).
      \]
      Here $\pi _E \colon E \to \pt$ is the structure morphism of $E$, 
      which is a local complete intersection morphism,  and 
      $\zeta_{E}: |E| \to E$ is the natural embedding.  
  \end{enumerate}
\end{defn}

\begin{rem}
  The axiom (QS5) is likely equivalent to the Extended Homotopy 
  Property.  In particular, by \cite[Chapter 7]{lm} it is satisfied by 
  all Borel--Moore homology functors.  As mentioned previously, it is 
  directly related to the universal morphism from regular bordism 
  commuting with locally complete intersection pullback (a priori, it 
  only commutes with smooth pullback).  Since we want our derived 
  bordism group to extend bordism (it will in fact be an isomorphism), 
  it is a requirement.  This hypothesis will likely be relegated 
  irrelevant with more investigation in properties of the ``naive'' 
  derived bordism groups.
\end{rem}

The pull-back $f^*$ is called an orientation of the quasi-smooth 
morphism. In any oriented Borel--Moore functor with quasi-smooth 
pull-backs or orientations we can define virtual fundamental classes.

\begin{defn}
  Let $X$ be a quasi-projective quasi-smooth derived scheme, and let 
  $A_*$ be an oriented Borel--Moore functor with quasi-smooth pull-backs.  
  Then the \emph{virtual fundamental class} is defined as
  \[
    \pi_X ^* ([1]) \in A_*(X).
  \]
  Here $\pi_X \colon X \to \pt$ is the structure morphism, and $\pi_X^*$ 
  is quasi-smooth pull-back.
\end{defn}

The following definition summarizes all desirable properties a homology 
theory with quasi-smooth pull-backs should have.
\begin{defn}
\label{defn:OBMHTqs}
  A \emph{oriented Borel--Moore homology theory with quasi-smooth 
    pullbacks}  on $\Cat{dQPr}_k$ is given by an additive functor $A_* 
  \colon \Cat{dQPr}'_k \to \Cat{Ab}_*$ equipped with quasi-smooth 
  pull-backs and an external product, such that   
  \begin{enumerate}[({BM}1)]
    \item The axioms (QS2), (QS 3) and (QS4) hold.
    \item The Projective Bundle Theorem of \cite[Definition 5.1.3]{lm} 
      holds.
    \item The Extended Homotopy relation of \cite[Definition 5.1.3]{lm} 
      holds.
    \item The Cellular Decomposition relation of \cite[Definition 
      5.1.3]{lm} holds.
  \end{enumerate}
\end{defn}

\section{Derived algebraic bordism}
\label{sect:dOm}

In this section we define derived algebraic bordism by generators and 
relations. Derived algebraic bordism will turn out to be the universal 
oriented Borel--Moore functor of geometric type with orientations for 
quasi-smooth morphisms. Since a quasi-smooth morphism of schemes is a 
local complete intersection morphism, derived algebraic bordism has 
pull-backs for local complete intersection morphisms. Besides all the 
axioms necessary for an oriented Borel--Moore functor of geometric type, 
derived algebraic bordism additionally satisfies the axiom (Loc) of 
Remark \ref{rem:loc}, although we will not use this.

We begin with the generators of our theory.
\begin{defn}
  \label{defn:generators}
  Let $X$ be a quasi-projective derived scheme over $k$. Denote by 
  $\sM_n (X)^{+}$ the free abelian group generated by proper morphisms
  \[
    f \colon Y \longrightarrow X
  \]
  where $Y \in \Cat{QSm}_k$ is irreducible and of virtual dimension $n$.  
  We will refer to elements of $\sM_*(X)$ as \emph{derived bordism 
    cycles}.
\end{defn}

We next introduce relations among the generators.
\begin{defn}
  \label{defn:relations}
  Let $X \in \Cat{dQPr}_k$, and denote by $p \colon X \times \IP^1 \to 
  \IP^1$ the projection onto the second factor. Let $Y \in \Cat{dQPr}_k$ 
  be quasi-smooth of pure virtual dimension, and let
  \[
    g \colon Y \to X \times \IP^1
  \]
  be a proper morphism. We can then form the homotopy Cartesian square
  \[
    \xymatrix{
      Y_0 \ar[r] \ar[d] & Y \ar[d]_{g} & Y_{\infty} \ar[d] \ar[l] \\
      X \ar[r] \ar[d] & X \times \IP^{1} \ar[d]_{p} & X \ar[d] \ar[l] \\
      0 \ar[r] & \IP^{1} & \ar[l] \infty.
    }
  \]
  The associated \emph{homotopy fiber relation} then is
  \[
    [Y_0 \to X] - [Y_{\infty} \to X].
  \]
  Let $\sR_{*}(X) \subset \sM_{*}(X)^+$ be the subgroup generated by all 
  homotopy fiber relations.
\end{defn}

\begin{rem}
  It follows from basic properties of homotopy fiber products that $Y_0$ 
  and $Y_{\infty}$ are quasi-smooth derived schemes without any 
  assumptions on $0$ and $\infty$ being regular values of $p \circ g$.
\end{rem}

We can now define a naive version of derived algebraic bordism.

\begin{defn}
  \label{defn:ndCob}
  Let $X \in \Cat{dQPr}_k$. Then \emph{naive derived algebraic bordism} 
  is defined by
  \[
    d\Omega^{\naive}_{*}(X)=\sM_{*}(X)^+/<\sR_{*}(X)>.
  \]
\end{defn}

\begin{rem}
  Let $[Y \to X]$ and $[Y' \to X]$ be generators of 
  $d\Omega^{\naive}_{*}(X)$ where $Y$ is weakly equivalent to $Y'$.  
  Using the homotopy fiber relation it follows that $[Y \to X]=[Y' \to 
  X]$ in $d\Omega^{\naive}_{*}(X)$.
\end{rem}
The functor $d\Omega_* ^{\naive}$ already has some of the 
requisite structure of an oriented Borel--Moore functor. In particular, 
we can immediately define  push-forward along proper maps and pullback 
along quasi-smooth morphisms (which includes smooth morphisms). Let $g 
\colon X \to X'$ be a proper map in $\Cat{dQPr}_k$. The map
\[
  g_* \colon \sM_{*}(X)^+ \longrightarrow \sM_*(X')^+
\]
is given by
\[
  g_*([f \colon Y \to X]) = [g \circ f \colon Y \to X'].
\]
It is immediate that this descends to a functorial push-forward
\[
  g_* \colon d\Omega_* ^{\naive}(X) \to d\Omega_* ^{\naive}(X').
\]


\begin{defn}
  \label{defn:pb}
  Let $g \colon X \to X'$ be a quasi-smooth morphism between 
  quasi-projective derived schemes. The \emph{quasi-smooth pull-back}
  \[
    g^* \colon \sM_{*}(X')^+ \longrightarrow \sM_{*}(X)^+
  \]
  is given by
  \[
    g^*([f \colon Y \to X']) = [Y \times^h_{X'} X \to X].
  \]
\end{defn}

\noindent Again, it is immediate that this descends to a functorial pull-back
\[
  g^* \colon d\Omega_* ^{\naive}(X') \longrightarrow d\Omega^{\naive}_* 
  (X).
\]

\begin{rem}
  In case $g$ is smooth and thus flat, the usual fiber product is a 
  homotopy fiber product and we arrive at
  \[
    g^*([f \colon Y \to X']) = [Y \times_{X'} X \to X].
  \]
\end{rem}

Since we have defined pull-backs for quasi-smooth morphisms we 
automatically have first Chern class operators.

\begin{defn}
  \label{defn:c1}
  Let $X \in \Cat{dQPr}_k$, $L$ a line bundle on $X$, and $s_0 \colon X 
  \to L$ the zero-section. We can then define the first Chern class 
  operator via
  \[
    c_1(L)([Y \to X]) = s_0^* (s_0)_* ([Y \to X]).
  \]
  More generally, for any vector bundle $E \to X$ of rank $r$ we define 
  the \emph{Euler class} or \emph{top Chern class} as
  \[
    c_r(L)([Y \to X]) = s_0^* (s_0)_* ([Y \to X]).
  \]
\end{defn}

\begin{lem}
\label{lem:secinvariant}
  Given a line bundle $L$  on $X \in \Cat{QPr}_k$ and any $s: X \to L$, 
  define $c_1^s(L): d\Omega_\ast(X) \to d\Omega_\ast(X)$ as $s_0^\ast 
  s_\ast ([Y \xrightarrow{f} X])$.  Then $c_1^s(L) = c_1^{s_0}(L) = 
  c_1(L)$.
\end{lem}
\begin{proof}
Let us first assume that $L$ has two non-zero sections $s$ and  
$s^\prime$.  Define $Z \to X \times \IP^1$ to be the derived scheme
\begin{displaymath}
  \xymatrix{Z \ar[r] \ar[d]^{\tilde f} & Y \times \IP^1 \ar[d]^{ \tilde 
      s \circ (f \times id)} \\ X \times \IP^1 \ar[r]^{s_0} & L 
    \boxtimes \sO_{\IP^1}(1)}
\end{displaymath}
where $\tilde s = sx_0 + s^\prime x_\infty$, and where $x_0, x_\infty$ 
are the Cartier divisors corresponding to $0$ and $\infty$ (they are 
both sections of $\sO_{\IP^1}(1)$).  It is clear that $Z$ is 
quasi-smooth and that $\tilde f|_{0} = c^s_1(L) \cap [Y \xrightarrow{f} 
X]$ and $\tilde f|_{\infty} = c^{s^\prime}_1(L) \cap [Y \xrightarrow{f} 
X]$, thus proving the claim.

For the case that $s = s_0$, note that if $p\colon L \to X$ is the natural projection, then $p^\ast L \to L$ has two natural sections: the canonical section $\tilde s$ and $p^\ast s^\prime$.  These sections are both non-zero sections of the same line bundle and thus by above, we have the following diagram for any $[X \xrightarrow{f} Y]$:

\begin{displaymath}
  \xymatrix{Y^{\prime \prime} \ar[r] \ar[d] & Y^{\prime} \ar[r] \ar[d] & 
    Y\times_X L \times \IP^1 \ar[d] \\  Z^\prime \ar[r] \ar[d] & Z 
    \ar[r] \ar[d] & L \times \IP^1 \ar[d] \\ X \times \IP^1 \ar[r]& L 
    \times \IP^1 \ar[r]& L\boxtimes \sO_{\IP^1} }
\end{displaymath}
The bottom morphisms are quasi-smooth, thus the quasi-smoothness of $Y$ 
implies $Y^{\prime \prime}$ is quasi-smooth. We then have
\[
  [Y^{\prime \prime}|_{0} \to X] = c_1^{s_0}(L) \cap [Y \to X]
\]
and
\[
  [Y^{\prime \prime}|_{\infty} \to X] = c_1^{s^\prime} (L) \cap [Y \to 
  X],
\]
and thus the desired bordism.
\end{proof}
\begin{rem}
  The same proof works also works for any vector bundle of rank $e \geq 
  1$.
\end{rem}

To promote $d\Omega_{*}^{\naive}$ to an oriented Borel--Moore functor 
with product we define the external product by \begin{align*}
  \times \colon \sM_*(X) \times \sM_*(X') &\longrightarrow \sM(X \times 
  X') \\
  [Y \to X] \times [Y' \to X'] &\longmapsto [Y \times _k Y' \to X 
  \times_k X'].
\end{align*}
Clearly this descends to $d\Omega_*^{\naive}$.

\begin{prop}
  \label{prop:obm}
  $d\Omega_* ^{\naive}$ is an oriented Borel--Moore functor with product.
\end{prop}
\begin{proof}
  We have defined the projective push-forward, smooth pull-back, first 
  Chern classes and the external product. The proof then is a long but 
  simple check of the axioms using properties of homotopy cartesian 
  squares. For instance, the axiom (A2) follows from the fact that given 
  a diagram
  \[
    \xymatrix{
      A' \ar[r] \ar[d] & A \ar[d]\\
      W \ar[r]^{g'} \ar[d]_{f'} & X \ar[d]^{f} \\
      Y \ar[r]_{g} & Z
    }
  \]
  where the two inner squares are homotopy cartesian, then the outer is 
  also homotopy cartesian.
\end{proof}

With these definitions, we can prove that $d\Omega_*^{\naive}$ partially satisfies the 
properties of an oriented Borel--Moore functor of geometric type.

\begin{prop}
  \label{prop:sect}
  $d\Omega_* ^{\naive}$ satisfies the axiom (Sect).
\end{prop}
\begin{proof}
  Let $X$ be a smooth scheme, $L$ a line bundle on $X$ and $s \colon X 
  \to L$ a section which is transverse to the zero-section. Since the 
  first Chern class operator is independent of the choice of section, we 
  can take $c_1(L)(1_X)=0^*s_*(1_X)$, where $0 \colon X \to L$ is the 
  zero-section. Since the section is assumed to be transverse to zero, 
  the first Chern class operator is then given by the Cartesian square
  \[
    \xymatrix{
      Z \ar[r] \ar[d]_{\id} & X \ar[d]^{\id}\\
      Z \ar[r] \ar[d]_{i} & X \ar[d]^{s} \\
      X \ar[r]_{0} & L
    }
  \]
  Here $Z$ is the zero-set of $s$. We now have
  \begin{align*}
    c_1(L)(1_X) &= 0^*s_*(1_X) = 0^*([X \overset{\id}{\longrightarrow} X 
    \overset{s}{\longrightarrow} L]) \\
    &= [Z \overset{\id}{\longrightarrow} Z \overset{i}{\longrightarrow} 
    X] = i_*([Z \overset{\id}{\longrightarrow} Z])
  \end{align*}
\end{proof}

\begin{rem}
  In fact, replacing $[X \to X]$ with $[Y \to X]$ in the proof above shows that $d\Omega_{*}^{\naive}$ satisfies the 
  axiom (Loc) of Remark \ref{rem:loc}.
\end{rem}

\begin{rem}
  Using the homotopy zero-set Proposition \ref{prop:sect} is even true without 
  the transverse to zero assumption.
\end{rem}

This is as far as one can get by using the homotopy fiber relation. It 
seems impossible to show further properties of the first Chern class 
operator of Definition \ref{defn:c1}, e.g., that it satisfies a formal 
group law and the axiom (Dim). To go any further, one must artificially 
impose a formal group law.  This requires that we first know $c_1$ acts 
nilpotently.

In the following propositions, following the arguments of \cite{lp}, we 
will show that it is enough to show nilpotence of $c_1$ for globally 
generated bundles. Once this is shown, it then is legal to impose the 
formal group law for globally generated bundles. We are then left to 
prove that the axioms (FGL) and (Dim) hold for all line bundles.

\begin{prop}
  \label{prop:nilp}
  Let $X$ be a smooth scheme, $L_1, \, \dots, L_r$ globally generated 
  bundles with $\dim(X) < r$. Then
  \[
    \prod_{i=1}^{r}c_1(L_i)=0
  \]
\end{prop}
\begin{proof}
  By Proposition \ref{prop:sect}, $c_1(L_1) \circ c_1(L_2) \circ \ldots 
  \circ c_1(L_r)(1_X)$ set-theoretically can be arranged to be the empty 
  set.  Any morphism factoring through the empty set is zero.
\end{proof}

We now impose the formal group law on these globally generated line 
bundles. Recall $\IL_*$ denotes the Lazard ring, graded such that the 
universal formal group law has total degree -1.  Define a new oriented 
Borel--Moore functor by
\[
  X \longmapsto \IL_{*} \otimes_{\mathbb{Z}} d\Omega_*^{\naive}(X).
\]
Set for any smooth scheme $X$
\[
  \sR^{\FGL}_{*}(X) \subset \IL_{*} \otimes_{\mathbb{Z}} d\Omega^{\naive}_{*}(X)
\]
to be the subset of elements of the form
\[
  F_{\IL}(c_1(L),c_1(M))(1_X) - c_1(L \otimes M)(1_X)
\]
for globally generated line bundles $L,M$ on a smooth scheme $X$. Here $F_{\IL}$ is the 
universal formal group law.

Recall from \cite[Section 2.1.5]{lm} that given an oriented Borel--Moore 
functor with product $A_*$ and a subset of homogeneous elements $\sR_*(X) 
\subset A_*(X)$ compatible with the external product it is possible to 
define the quotient Borel--Moore homology functor $A_* / \sR_*$.

\begin{defn}
  \label{defn:dCob}
  Define \emph{derived algebraic pre-bordism} as
  \[
    d\Omega^{\Pre}_{*} = d\Omega^{\naive}_{*} / <\sR^{\FGL}_{*}>.
  \]
\end{defn}

Using the notation for the formal group law set in the previous section, we now have to show that (Dim) and (FGL) hold for all line bundles.  The proofs proceed exactly as in \cite{lp}.

\begin{lem}
  \label{lem:c1L}
  Let $X \in \Cat{Sm}_k$, $L$ a line bundle on $X$ and $M$ a globally 
  generated line bundle such that $L \otimes M$ is globally generated.  
  Then
  \[
    c_1(L) = F^{-}_{\IL}(c_1(L \otimes M), c_1(M)).
  \]
\end{lem}
\begin{proof}
  Follows from the power series identity
  \[
    F^{-}(F(u,v),F(0,v))=F^{-}(u,0)=u
  \]
\end{proof}

\begin{prop}
  \label{prop:dim}
  $d\Omega_*^{\Pre}$ satisfies the axiom Dim.
\end{prop}
\begin{proof}
  Let $X$ be a quasi-projective scheme, $L_1, \, \dots L_r$ a sequence 
  of line bundles with $r > \dim(X)$. By the previous lemma, we have
  \[
    c_1(L_i) = F^{-}_{\IL}(c_1(L_i \otimes M), c_1(M)).
  \]
  where $M$ is globally generated and chosen such that $L_i \otimes M$ 
  is globally generated. The proof then is the same as  \cite[Lemma 
  9.3]{lp}.
\end{proof}

\begin{prop}
  \label{prop:fgl}
  $d\Omega_*^{\Pre}$ satisfies the axiom (FGL).
\end{prop}
\begin{proof}
  Let $X$ be a smooth scheme, and let $L,M$ be line bundles on $X$.  
  Choose globally generated bundles $N_1,N_2$ such that $L \otimes N_1$ 
  and $L \otimes N_2$ are globally generated. The proof then follows 
  from the formal power series identity
  \[
    F^{-} \left( F(u_1,v_1),F(u_2,v_2) \right) = F \left( 
      F^{-}(u_1,u_2),F^{-}(v_1,v_2) \right)
  \]
  combined with Lemma \ref{lem:c1L} and the formal group law for globally 
  generated bundles.
\end{proof}

Lastly, we ensure strict normal crossing divisors in a smooth 
scheme have the correct fundamental class.  Set for any smooth scheme $X$ and any strict normal crossing divisor $E$ 
in $X$
\[
  \sR^{\SNC}_{*}(E) \subset  d\Omega^{\Pre}_{*}(E)
\]
to be the subset of elements of the form
\[
  \pi_E^*([1]) - (\zeta_{E})_\ast[E \to |E|].
\]
Here $\pi_E \colon E \to \pt$ is the structure morphism of the support 
of $E$, $\zeta_{E}: |E| \to E$ is the natural closed embedding, and 
$[E \to |E|]$ is the class defined in \eqref{eq:FundSNC}.
\begin{defn}
  Define \emph{derived algebraic bordism} as
  \[
    d\Omega_* = d\Omega_*^{\Pre} / <\sR_*^{SNC}>.
  \]
\end{defn}

\begin{rem}
  Instead of taking the axiomatic approach to obtaining a formal group 
  law it should be possible to impose some form of the double point 
  relations of \cite{lp} to obtain the formal group law and the correct 
  fundamental class for strict normal crossing divisors in one step.
\end{rem}

We are now ready to prove the universality of derived algebraic bordism.

\begin{thm}
  \label{thm:universal}
  $d\Omega_{*}$ is the universal oriented Borel--Moore functor with 
  quasi-smooth pullbacks of geometric type.
\end{thm}
\begin{proof}
  By Proposition \ref{prop:obm}, \ref{prop:sect}, \ref{prop:dim} and 
  \ref{prop:fgl} $d\Omega_*$ is an oriented Borel--Moore functor of 
  geometric type and has orientations for quasi-smooth morphisms. The 
  normalization property (QS5) is clear by construction. We are left to 
  show universality.  The proof of universality is the same as 
  \cite[Proposition 1.10]{quillen}.  Let $A_*$ be another oriented 
  Borel--Moore functor of geometric type with quasi-smooth pullbacks.  
  Then we can define a natural transformation
  \[
    \vartheta_{A} \colon d\Omega_* \longrightarrow A_*
  \]
  by
  \[
    \vartheta_{A} ([Y \overset{f}{\longrightarrow} X]) = f_*\pi_Y^* [1].
  \]
  Here as before $\pi_Y \colon Y \to \pt$ is the structure morphism of 
  $Y$. The proof that $\vartheta$ is compatible with the structures of 
  an oriented Borel--Moore functor of geometric type is straightforward.
\end{proof}

\begin{rem}
  Further examples of oriented Borel--Moore functors with quasi-smooth 
  pullbacks are discussed in \cite{ls}. There, using the quasi-smooth 
  pull-backs defined by \cite{manolache}, Chow homology is extended to 
  derived schemes. This theory has the additive formal group law. As 
  another example, $G$-theory with quasi-smooth pull-backs is introduced.  
  This carries the multiplicative formal group law.
\end{rem}

\section{Quasi-smooth pullbacks in algebraic bordism}
\label{sect:qSmPb}

We now restrict to the case that $k$ is of characteristic zero.  We 
begin by extending classical, underived bordism $\Omega_*$ to a functor 
on the category of derived quasi-projective schemes $\Cat{dQPr}_k$ with 
proper morphisms using the same generators and relations as for 
classical schemes. By abuse of notation, we will still denote this 
extended functor by $\Omega_*$.  Given a derived quasi-projective scheme 
$X$, we then have the closed embedding of the underlying scheme 
$\iota_{X} \colon t_0(X) \hookrightarrow X$. In particular, this is a 
proper map and there exists a push-forward on associated bordism groups.

\begin{lem}
  Let $X \in \Cat{dQPr}_k$. Then the inclusion of the classical part 
  $\iota_X \colon t_0(X) \hookrightarrow X$ induces an isomorphism
  \[
    \iota_{X*} \colon \Omega_*(t_0(X)) \longrightarrow \Omega_*(X).
  \]
\end{lem}
\begin{proof}
  This is clear, since the generators of $\Omega_*(X)$ are given by 
  morphisms $f \colon Y \to X$ with $Y$ smooth, and every such map 
  factors through the truncation $t_0(X)$.
\end{proof}

Using the above lemma, it is immediate to equip the extended functor 
$\Omega_*$ with pull-backs along smooth morphisms of derived 
quasi-projective schemes. Since the truncation of a smooth morphism $f 
\colon X \to Y$ is again smooth, we can define $f^*$ by the composition
\[
  \Omega_*(Y) \cong \Omega_*(t_0 (Y)) \overset{t_0(f)^*}{\to} 
  \Omega_*(t_0(X)) \cong \Omega_*(X).
\]

Along the same lines, we can define a first Chern class operator for a 
line bundle $L$ on a derived scheme $X$ for the extended functor 
$\Omega_*$ by the composition
\[
  \Omega_*(X) \cong \Omega_{*}(t_0(X)) \overset{c_1(\iota_X^* L)}{\to} 
  \Omega_{*-1}(t_0(X)) \cong \Omega_{*-1}(X).
\]

Almost by definition, this makes the extended functor $\Omega_*$ to an 
oriented Borel--Moore funtor of geometric type. The same proof as in the 
discrete case shows that it is in fact the universal such functor.

\begin{cor}
  \label{cor:OmToDOm}
  There is a canonical classifying morphism
  \[
    \vartheta_{d\Omega} \colon \Omega_{*} \longrightarrow d\Omega_{*}
  \]
\end{cor}
\begin{proof}
  Since $d\Omega_{*}$ is an oriented Borel--Moore functor of geometric 
  type on $\Cat{dQPr}_k$, and since by the preceeding discussion  
  $\Omega_{*}$ is the universal such Borel--Moore functor of geometric 
  type, we obtain a natural transformation
  \[
    \vartheta_{d\Omega} \colon \Omega_{*} \longrightarrow d\Omega_{*}
  \]
  given by
  \[
    \vartheta_{d\Omega}([Y \overset{f}{\longrightarrow} X]) = f_* \pi_Y 
    ^* [1].
  \]
  Here $\pi_Y \colon Y \to \pt$ is the structure morphism of $Y$, and 
  the pull-back and push-forward morphisms on the right hand side are in 
  $d\Omega_*$.
\end{proof}

The remainder of this section is devoted to constructing a classifying 
map in the opposite direction. To construct this morphism, we want to 
use the universal property of $d\Omega_*$. Since $d\Omega_*$ is 
universal with respect to quasi-smooth pullbacks, we first have to 
construct quasi-smooth pull-backs in $\Omega_*$. This will be done in 
the this section.

More generally, it is possible to show this for any Borel--Moore 
homology theory that has intersections with pseudo-divisors defined (in 
such a way as to commute with smooth pull-back and proper push-forward) 
and satisfies a homotopy invariance property.  

In order to construct these quasi-smooth pullbacks, we first review some 
background material on quasi-smooth embeddings.

\subsection{Background on quasi-smooth embeddings}

Let $f \colon X \hookrightarrow Y$ be a quasi-smooth closed embedding.  
Using the truncation functor, we have a commutative diagram in 
$\Cat{dSch}_k$
\[
  \xymatrix{
    X \ar@{^{(}->}[r]^f & Y \\
    t_0(X) \ar@{^{(}->}[u]^{\iota_X} \ar@{^{(}->}[r]_{t_0(f)} & t_0(Y) 
    \ar@{^{(}->}[u]_{\iota_Y}
  }
\]

We have the following basic result.
\begin{lem}
  \label{lem:N}
  Let $f \colon X \hookrightarrow Y$ be a quasi-smooth embedding of 
  virtual co-dimension $d$. Then $\iota_X^*L_{X/Y}[-1]$ is a locally free sheaf of rank $d$ on 
  $t_0(X)$.
\end{lem}
\begin{proof}
  Since $f$ is quasi-smooth, $L_{X/Y}$ is perfect of Tor-amplitude $\leq 
  1$. Furthermore, since $f$ is an embedding, $L_{X/Y}$ is in 
  $\QCoh(X)_{\geq 1}$ (homological grading). It thus follows that $\iota_X^*L_{X/Y}[-1]$ is a 
  locally free sheaf. The claim on the rank is clear from the definition 
  of virtual co-dimension.
\end{proof}

\begin{defn}
  \label{defn:VNB}
  Let $f \colon X \hookrightarrow Y$ be a quasi-smooth embedding.  
  Define the \emph{virtual normal bundle} $N_X Y \to t_0(X)$ to be the 
  geometric vector bundle corresponding to $\iota^*_X 
  L_{X/Y}[-1]^{\vee}$.
\end{defn}

We now compare the virtual normal bundle $N_X Y$ to the normal 
cone $C_X Y$ of the underlying embedding $t_0(f) \colon t_0(X) 
\hookrightarrow t_0(Y)$. Recall that the normal cone is the scheme over 
$t_0(X)$ given by $\Spec_{\sO_{t_0(X)}} (\oplus_{n \geq 0} 
I^n/I^{n+1})$, where $I$ is the ideal sheaf of $t_0(X)$ in $t_0(Y)$.

Using the functoriality properties of the cotangent complex, we have a 
morphism
\[
  \iota_X^* L_{X/Y} \longrightarrow L_{t_0(X)/t_0(Y)}.
\]
Let $I$ be the ideal sheaf of $t_0(X)$ in $t_0(Y)$. It is a classical 
fact that $\pi_1 (L_{t_0(X)/t_0(Y)})=I/I^2$. We thus obtain a morphism
\[
  i \colon C_{X}Y \to N_X Y.
\]

We have the following result which will be the basis for the 
construction of quasi-smooth pullbacks in $\Omega_*$.

\begin{lem}
  \label{lem:CtoNclosed}
  The morphism
  \[
    i \colon C_X Y \longrightarrow N_X Y
  \]
  is a closed embedding.
\end{lem}
\begin{proof}
  The claim is local and will follow once it is shown that $j_X^* 
  L_{X/Y} \longrightarrow I/I^2$ is surjective . Let $f \colon A \to B$ 
  be a 0-connective morphism of simplicial commutative $k$-algebras. We 
  then have the commutative diagram
  \[
    \xymatrix{
      A \ar[r] \ar[d] & B\ar[d] \\
      \pi_0 A \ar[r] & \pi_0 B.
    }
  \]
  This gives the following diagram of cofiber sequences
  \[
    \xymatrix{
      L_A \otimes _A \pi_0 B \ar[r] \ar[d] & L_B \otimes _B \pi_0 B 
      \ar[r] \ar[d]& L_{B/A} \otimes_B \pi_0 B \ar[d] \\
      L_{\pi_0 A} \otimes _{\pi_0 A} \pi_0 B \ar[r] \ar[d] & L_{\pi_0 B} 
      \ar[r] \ar[d] & L_{\pi_0 B / \pi_0 A} \ar[d] \\
      L_{\pi_0 A /A} \otimes _{\pi_0 A} \pi_0 B \ar[r] & L_{\pi_0(B)/B} 
      \ar[r] & M.
    }
  \]
  We have to show that $M \in \QCoh(\pi_0 B)_{\geq 2}$. Now recall that 
  since $A \to \pi_0 A$ is 1-connective, $L_{\pi_0 A /A}$ is in 
  $\QCoh(\pi_0 A)_{\geq 2}$. Likewise $L_{\pi_0 B /B}$ is 2-connective. 
  Since $M$ is the cofiber of 2-connective objects the claim follows.
\end{proof}

\subsection{Deformation to the normal cone}
\label{sec:DNC}

We now introduce deformation to the normal cone for a quasi-smooth 
embedding $X \hookrightarrow Y$.

We begin by reviewing the Rees construction for filtered simplicial 
modules over a field $k$. Given a simplicial $k$-module $M$ with 
filtration $\{M_i\}_{i \in \IZ}$ such that $M_i \subset M_{i+1}$ and 
$M=\cup M_i$, we can form a simplicial graded $k[t]$-module $bigoplus 
M_i t^i$. We call this association the Rees construction.

We now review some facts from commutative algebra of simplicial rings 
\cite{quillen2}. Let $f \colon R \to S$ be a morphism of simplicial 
commutative $k$-algebras which is levelwise surjective. We equip the 
categories of simplicial commutative $k$-algebras and simplicial 
$k$-modules with their standard model structures. We then obtain a fibre 
sequence
\[
  I \to R \to S
\]
of simplicial $k$-modules where $I$ is a simplicial ideal in $R$. Now 
factor this morphism in the model category of simplicial commutative 
$k$-algebras as $R \to P \to S$ where $R \to P$ is a cofibration and $P 
\to S$ is a weak equivalence.  Setting $Q = P \otimes_{R} S$, we obtain 
a multiplication map $m \colon Q \to S$.  This gives a fiber sequence
\[
  J \to Q \to S
\]
in the category of simplicial $k$-modules where $J$ is a simplicial 
ideal in $Q$. The cotangent complex $L_{S/R}$ can then be identified 
with $J/J^2$ in the homotopy category of simplicial $S$-modules.

By identifying $S$ with $S \otimes_{R} R$, it follows that the 
multiplication $m \colon Q \to S$ has a section $s = \id \otimes f$ in 
the homotopy category of simplicial $Q$-modules. In particular, we 
obtain an identification $J \simeq I \otimes_{R}^{\IL} S [1]$ as 
$Q$-modules. By adjunction, we have an identification as $S$-modules $J 
\otimes^{\IL}_Q S \simeq I \otimes_{R}^{\IL} S [1]$. In particular, we 
have $J/J^2 = I/I^2[1]$.  Combining with the formula for the contangent 
complex in the above paragraph, we obtain the identification $L_{S/R} = 
I/I^2[1]$.

Now  assume that $A$ is a smooth discrete simplicial $k$-algebra and $f 
\colon A \to B$ is quasi-smooth morphism which induces a surjective 
morphism $\pi_0 A \to \pi_0 B$. By first factoring $A \to B$ in the 
model category of simplicial $k$-algebras as $A \to B' \to B$, where $A 
\to B'$ is a cofibration and $B' \to B$ is a weak equivalence, and then 
factoring $A \to B'$ as $A \to A' \to B'$, where $A \to A'$ is a 
cofibration and a weak equivalence and $A' \to B'$ is a fibration, we 
can assume that $f' \colon A' \to B'$ is levelwise surjective, and $A'$ 
and $B'$ are both levelwise smooth $k$-algebras. In particular, in the 
fibre sequence
\[
  I \to A' \overset{f'}{\to} B'
\]
the simplicial ideal $I$ is levelwise a regular ideal. Furthermore, 
since $A' \to B'$ is quasi-smooth, using the identification 
$L_{B'/A'}[-1] \simeq I/I^2$ obtained above it follows that $I/I^2$ is a 
projective $A'$-module.

We now filter $A'$ by powers of the simplicial ideal $I$ to obtain a 
filtered simplicial $k$-algebra $(A', F)$. Applying the 
Rees-construction gives a simplicial graded $k[t]$-algebra $\bigoplus 
_{n \in \IZ} F^i t^i$. The fibre over zero can be identified with the 
associated graded simplicial algebra $\bigoplus _{n \geq 0} 
I^n/I^{n+1}$. Since the simplicial ideal $I$ is levelwise regular and 
$I/I^2$ is projective, we obtain an identification $\Sym^*_B (I/I^2) 
\simeq I^n/I^{n+1}$. As a consequence, we can identify the fibre over 
zero of the Rees-algebra $\bigoplus _{n \in \IZ} F^i t^i$ with $\Sym^*_B 
(L_{B/A}[-1])$.

The previous discussion immediately generalizes to schemes, giving us a 
deformation to the normal cone space $M_X^{\circ}Y$. By the fibre-wise 
criterion for flatness, this space is flat over $Y \times \IA^1$. The 
fibre over zero of the truncation $t_0 (M_X^{\circ}Y)$ is a Cartier 
divisor on $t_0(M_X^{\circ}Y)$.  In case $X \hookrightarrow Y$ is a 
quasi-smooth embedding, the fibre over zero is given by the vector 
bundle corresonding to the locally free sheaf $j_X ^* L_{X/Y}[-1]$, and 
thus is the virtual normal bundle $N$ of Definition \ref{defn:VNB}.

\subsection{Pull-back along quasi-smooth embeddings}

%
%

We will define the pull-back along a quasi-smooth embedding $f\colon X 
\hookrightarrow Y$ using the deformation to the normal cone space 
introduced in Section \ref{sec:DNC}. We will abbreviate the space 
$t_0(M^{\circ}_{X}Y)$ to $M^{\circ}$. Recall that this scheme is flat 
over $Y \times \IA^1$, with the virtual normal bundle $N$ of Definition 
\ref{defn:VNB} as fibre over zero.

Applying the localization sequence for algebraic bordism in 
characteristic zero, we obtain an exact sequence
\[
  \Omega_*(N) \overset{i_*}{\longrightarrow} \Omega_*(M^{\circ}) 
  \overset{j^*}{\longrightarrow} \Omega_*(Y \times (\IA^1 \setminus 0)) 
  \longrightarrow 0.
\]
Since $i^*i_*=0$, we thus obtain a morphism
\[
  s_{X/Y} \colon \Omega_*(Y \times (\IA^1  \setminus 0)) \longrightarrow 
  \Omega_{*-1}(N).
\]
Finally, we can define the specialization morphism as the composition
\[
  \sigma_{X/Y} \colon \Omega_*(Y) \overset{\pr^*}{\longrightarrow} 
  \Omega_{*+1} (Y \times \IA^1) \overset{s_{X/Y}}{\longrightarrow} 
  \Omega_* (N).
\]

In complete analogy to the case of pull-backs along regular embeddings, 
we can then define pull-backs along quasi-smooth morphisms.

For the following definition, recall that since algebraic cobordism 
satisfies the extended homotopy property, for any scheme $X$ over $k$ 
and any rank $n$ vector bundle $p \colon E \to X$, the morphism $p^* 
\colon \Omega_*(X) \to \Omega_{*+n}(E)$ is an isomorphism. This allows 
us to form the inverse $(p^*)^{-1}$. In particular, we can apply this to 
the virtual normal bundle. Note that the same reasoning does not apply 
if the normal cone fails to be a vector bundle. This is precisely what 
makes it impossible to define pullbacks along embeddings that are not 
local complete intersections or more generally quasi-smooth.

\begin{defn}
  \label{defn:QsPbOm}
  Let $f \colon X \hookrightarrow Y$ be a quasi-smooth embedding of 
  virtual codimension $d$ with $Y \in \Cat{QPr}_k$, and let $p \colon N 
  \to t_0(X)$ be the projection of the virtual normal bundle.  We define 
  quasi-smooth pullback as the composition
  \[
    f^* \colon \Omega_*(Y) \overset{\sigma_{X/Y}}{\longrightarrow} 
    \Omega_*(N) \overset{(p^*)^{-1}}{\longrightarrow} \Omega_{*-d} 
    (t_0(X)) \overset{(\iota_X)_*}{\longrightarrow} \Omega_{*-d}(X).
  \]
\end{defn}

\begin{rem}
  \label{rem:AfterLift}
  Unraveling the definition of $\sigma_{X/Y}(u)$ for a class $u \in 
  \Omega_*(Y)$, we see that we first have to pull back $u$ to $Y \times 
  (\IA^1 \setminus 0)$ giving us a class $\pr^* u \in \Omega_{*+1}(\IA^1 
  \setminus 0)$, then choosing a preimage $\tilde{u}$ of $\pr^* u$ in 
  $\Omega_{*+1}(M^{\circ})$, and finally intersecting with the Cartier 
  divisor $N$ to obtain $i_N^*\tilde{u} \in \Omega_*(N)$.

  Thus once we have chosen a lifting $\tilde{u} \in 
  \Omega_{*+1}(M^{\circ})$, quasi-smooth pull-back is given by 
  $\iota_X)_* \circ (p^*)^{-1} \circ i_{N}^*$. This relates the general 
  quasi-smooth pull-back to intersecting with an effective Cartier 
  divisor and intersecting with a zero-section.
\end{rem}

We now relate this to an alternative definition, which is closer to the 
definition of virtual pull-backs of \cite{manolache}. Denote by 
$M^{\circ}_{t_0(X)}Y$ the deformation to the normal cone space of the 
inclusion $t_0(X) \hookrightarrow Y$. The fiber over 0 of 
$M^{\circ}_{t_0(X)}Y$ is $C_{t_0(X)} Y$, the normal cone of $t_0(X)$ in 
$Y$.  This is an effective Cartier divisor on $M^{\circ}_{t_0(X)}Y$.  
Again using that $i^* i_*=0$ for this divisor and the exact localization 
sequence for algebraic bordism, there is a specialization morphism
\[
  \sigma_{t_0(X)/Y} \colon \Omega_*(Y) \longrightarrow \Omega_{*+1}(Y 
  \times (\IA^1  \setminus 0)) \longrightarrow \Omega_*(C_{t_0(X)} Y).
\]

By Lemma \ref{lem:CtoNclosed}, we have a closed immersion $j \colon 
C_{t_0(X)} Y \hookrightarrow N_X Y$ of the normal cone of $t_0(X)$ in 
$Y$ into the virtual normal bundle.

\begin{lem}
  Let $f \colon X \to Y$ be a quasi-smooth embedding. Denote by $j 
  \colon C_{t_0(X)}Y \to N_X Y$ the inclusion of Lemma 
  \ref{lem:CtoNclosed}. Then
  \[
    \sigma_{X/Y} = j_* \circ \sigma_{t_0(X)/Y}.
  \]
\end{lem}
\begin{proof}
  By \cite[Lemma 6.2.1]{lm}, intersection with a Cartier divisor 
  commutes with proper push-forward.
\end{proof}
\begin{cor}
 \label{cor:EquivToManDef}
 Let $f \colon X \hookrightarrow Y$ be a quasi-smooth embedding. Then
  \[
    f^*(u) = \iota_{X*} \circ (p^*)^{-1} \circ j_* \circ 
    \sigma_{X/Y}(u).
  \]
  for all $u \in d\Omega_*(Y)$.
\end{cor}

\begin{rem}
  By Corollary \ref{cor:EquivToManDef}, the pull-back for quasi-smooth 
  embeddings differs from the pull-back for regular embeddings only by 
  the inclusion of the normal cone into the virtual normal bundle.
\end{rem}

The formula for quasi-smooth pull-backs given in Corollary 
\ref{cor:EquivToManDef} is the most convenient form to prove the 
expected properties of the orientations defined. The remainder of this 
section closely follows the construction of pull-backs for local 
complete intersections of Verdier \cite{verdier}. We first recall some 
basic properties of the specialization homomorphism.

\begin{lem}
  \label{lem:SpecProp}
  Let
    \[
      \xymatrix{
        X' \ar@{^{(}->}[r]^{i'} \ar[d]_{f'} & Y' \ar[d]^{f} \\
        X \ar@{^{(}->}[r]^{i} & Y
      }
    \]
    be a Cartesian diagram in $\Cat{QPr}_k$ with $i$ a closed embedding.  
    Let $g \colon C_{X'}Y' \to C_X Y$ be the induced morphism of normal 
    cones.
  \begin{enumerate}[(a)]
    \item Assume that $f$ in the above diagram is proper. Then
      \[
        \xymatrix{
          \Omega_*(Y') \ar[r]^{\sigma_{X'/Y'}} \ar[d]_{f_*} & 
          \Omega_*(C_{X'} Y') \ar[d]_{g_*} \\
          \Omega_*(Y) \ar[r]_{\sigma_{X/Y}} & \Omega_* (C_{X} Y)
        }
      \]
      commutes.

    \item  Assume that $f$ in the above diagram is smooth of pure 
      relative dimension $d$.  Then 
      \[
        \xymatrix{
          \Omega_*(Y) \ar[r]^{\sigma_{X/Y}} \ar[d]_{f^*} & \Omega_*(C_X 
          Y) \ar[d]^{g^*} \\
          \Omega_{*+d}(Y') \ar[r]_{\sigma_{X'/Y'}} & \Omega_{*+d} 
          (C_{X'} Y')
        }
      \]
      commutes.
  \end{enumerate}
\end{lem}
\begin{proof}
  By functoriality of the deformation to the normal cone space we have a 
  morphism $h \colon M^{\circ}_{X'}Y' \to M^{\circ}_{X}Y$ which induces 
  $g \colon C_{X'}Y' \to C_X Y$ on the exceptional divisors. 

Assume that $f$ is proper. Intersection with a pseudo-divisor 
  commutes with proper push-forward \cite[Lemma 6.2.1]{lm}, so the diagram
  \[
    \xymatrix{
      \Omega_*(Y' \times (\IA^1 \setminus 0)) \ar[r]^(.6){\sigma^*_{X'/Y'}} 
      \ar[d] & \Omega_*(C_{X'}Y') \ar[d]^{g_*} \\
      \Omega_*(Y \times (\IA^1 \setminus 0)) \ar[r]_(.6){\sigma^*_{X/Y}} & 
      \Omega_*(C_{X}Y) \\
    }
  \]
  commutes, and the claim follows.

For the case that $f$ is smooth, one can explicitly compute that 
$M^{\circ}_{X^\prime}Y^\prime= Y^\prime \times \IA^1 \times_{Y \times 
  \IA^1} M^{\circ}_{X} Y$ and $h$ is the natural projection.  In 
particular, by base change, the morphism $g$ is smooth.  The proof now 
proceeds as above since intersection with a pseudo-divisor commutes with 
smooth pull-back \cite[Lemma 6.2.1]{lm}.  \end{proof}

\begin{rem}
  The previous Lemma \ref{lem:SpecProp} also admits an indirect proof.  
  Assume there exists some Cartesian diagram such that the diagrams in 
  (a) and (b) do not commute. Then one has an immediate contradiction to 
  the functoriality properties of the refined pull-backs of $\Omega_*$ 
  stated in \cite[Proposition  6.6.3]{lm}.
\end{rem}

\begin{lem}
  \label{lem:QSPushPull}
  Let
  \[
    \xymatrix{
      X' \ar@{^{(}->}[r]^{i'} \ar[d]_{f'} & Y' \ar[d]^{f} \\
      X \ar@{^{(}->}[r]^{i} & Y
    }
  \]
  be a homotopy Cartesian diagram in $\Cat{dQPr}_k$ with $i$ a 
  quasi-smooth closed embedding.
  \begin{enumerate}[(a)]
       \item Assume that in the above diagram $f$ is proper. Then
      \[
        i^*f_* = f'_*i^{'*}.
      \]
       \item Assume that in the above diagram $f$ is smooth. Then
      \[
        i^{\prime *} f^* = f^{\prime *} i^*.
      \]
  \end{enumerate}
\end{lem}
\begin{proof} 
  Since the truncation functor $t_0$ takes homotopy Cartesian diagrams 
  to Cartesian diagrams, we can apply Lemma \ref{lem:SpecProp} to the 
  diagram
  \[
    \xymatrix{
      t_0(X') \ar@{^{(}->}[r]^{t_0(i')} \ar[d]_{t_0(f')} & t_0(Y') 
      \ar[d]^{t_0(f)} \\
      t_0(X) \ar@{^{(}->}[r]^{t_0(i)} & t_0(Y).
    }
  \]
  Using basic functoriality properties of the cotangent complex we 
  obtain a further commutative diagram 
  \[
    \xymatrix{
      C_{X'}Y' \ar[r] \ar[d]_g & N_{X'} Y' \ar[d] \\
      C_X Y \ar[r] & N_X Y.
    }
  \]
  Here $g$ is the morphism of normal cones induced by base-change used 
  in Lemma \ref{lem:SpecProp}.

  Now assume that $f$ is proper. Combining Lemma \ref{lem:SpecProp} with 
  the above diagram it follows that
  \[
    \xymatrix{
      \Omega_*(Y') \ar[r]^{\sigma_{X'/Y'}} \ar[d]_{f_*} & 
      \Omega_*(C_{X'/Y'}) \ar[r] \ar[d]^{g_*} & \Omega_*(N_{X'}Y') 
      \ar[r] \ar[d] & \Omega_*(X) \ar[d]^{f'_*}\\
      \Omega_*(Y) \ar[r]_{\sigma_{X/Y}} & \Omega_*(C_{X/Y}) \ar[r] & 
      \Omega_*(N_{X} Y) \ar[r]& \Omega_*(X)
    }
  \]
  commutes. Since the composition of the horizontal morphisms is 
  respectively $i^*$ and $i^{\prime *}$ the claim follows. The case of 
  $f$ smooth is analogous.
\end{proof}

Our next task is to verify functoriality of the orientation for 
quasi-smooth closed embeddings. Apart from the fact that we have to keep 
track of the embeddings of the normal cones in the virtual normal 
bundles, the proof proceeds in exact analogy to the case of regular 
embeddings of discrete schemes.

\begin{prop}
  \label{prop:FuncQSE}
  Let $X \overset{i}{\hookrightarrow}Y \overset{j}{\hookrightarrow} Z$ 
  be quasi-smooth embeddings. Then
  \[
    (j \circ i)^* = i^*j^*
  \]
\end{prop}
\begin{proof}
  Let us first assume that $t_0(j) \colon t_0(Y) \hookrightarrow t_0(Z)$ 
  is given by the embedding
  \[
    t_0(Y) \hookrightarrow C_Y Z \hookrightarrow N_Y Z
  \]
  of $Y$ in the virtual normal bundle. Let $p \colon N_Y Z \to t_0(Y)$ 
  be the projection. We have an isomorphism  \cite[Proof of Theorem 
  6.5]{fulton}
  \[
    C_X (N_Y Z) \simeq C_X Y \times _{t_0(X)} (N_Y Z \times _{t_0(Y)} 
    t_0(X))
  \]
  inducing a morphism
  \[
    q' \colon C_X (N_Y Z) \longrightarrow C_X Y.
  \]
  Additionally assume that the virtual normal bundle $N_X Z$ splits as 
  $N_Y Z \oplus N_X Y$, and let $q \colon N_X Y \oplus N_Y Z \to N_X Y$ 
  be the projection, giving us a commutative diagram
  \[
    \xymatrix{
      C_X (N_Y Z) \ar[r] \ar[d]^{q'} & N_X Y \oplus N_Y Z \ar[d]_q\\
      C_X Y \ar[r] & N_X Y
    }
  \]
   Here the lower horizontal morphism is the canonical inclusion of the 
   normal cone in the virtual normal bundle.  In this situation, we have 
   an isomorphism of the open deformation to the normal cone spaces 
   $M^{\circ}_{t_0(X)/N_Y Z} \cong M^{\circ}_{t_0(X)/t_0(Y)} \times 
   _{t_0(Y)} N_Y Z$ compatible with the projections to $\IA^1$ 
   \cite[Corollaire 2.18]{verdier}. We thus obtain a commutative diagram
  \[
    \xymatrix{
      \Omega_*(Y) \ar[r]^{\sigma_{X/Y}} \ar[d]_{p^*} & \Omega_{*} (C_X 
      Y) \ar[r] \ar[d]_{q^{'*}}& \Omega_* (N_X Y) \ar[d]_{q^*} \\
      \Omega_{*+d}(Z) \ar[r]_{\sigma_{X/Z}} & \Omega_{*+d}(C_X Z) \ar[r] 
      & \Omega_{*+d}(N_X Y \oplus N_Y Z)
    }
  \]
  To prove our claim, since $j^*=(p^*)^{-1}$ it suffices to prove
  \[
    (j \circ i)^* p^* = i^*.
  \]
  Let $\pi \colon N_X Y \to X$ and $\rho \colon N_X Y \oplus N_Y Z \to 
  X$ be the structure morphisms.  Since $q^* \circ \pi^*=\rho^*$ is an 
  isomorphism, it suffices to prove
  \[
    (\pi \circ q)^* \circ (j \circ i)^* \circ p^* = q^* \circ \pi^* 
    \circ i^*.
  \]
  But this is immediate from the commutativity of the above diagram. The 
  reduction of the general case to the special case treated here is a 
  standard application of the double deformation space
  \[
    M^{\circ}_{X \times \IA^1 / M^{\circ}_{Y/Z}} \longrightarrow \IA^1 
    \times \IA^1.
  \]
  obtained by applying the deformation to the normal cone construction 
  to the embedding $X \times \IA^1 \hookrightarrow M^{\circ}_{Y/Z}$
  as in \cite[Theorem 4.8]{manolache},  \cite[Theorem 6.6.5]{lm} or 
  \cite[Theoreme 4.4]{verdier}.
\end{proof}

\subsection{Pull-backs for quasi-smooth morphisms}

After having defined orientations for quasi-smooth closed embeddings we 
now move on general quasi-smooth morphisms. Since we are only working 
with quasi-projective derived schemes a smoothing is always 
available.

\begin{lem}
  Let $f \colon X \to Y$ be a quasi-smooth morphism in $\Cat{dQPr}_k$.  
  Assume that $f = p \circ i = p' \circ i'$ are two smoothings of 
  $f$. Then
  \[
    i^*p^*=i^{'*}p^{'*}.
  \]
\end{lem}

\begin{proof}
  Let $X \overset{i}{\hookrightarrow}P \overset{p}{\to}Y$ and $X 
  \overset{i'}{\hookrightarrow}P' \overset{p'}{\to}Y$ be the two 
  smoothings. Using the diagonal embedding $X \to P \times_Y P'$ and the 
  cartesian diagram
  \[
    \xymatrix{
      P' \ar[r] \ar[d]& P \times_Y P' \ar[d]\\
      X \ar[r] & P
    }
  \]
  one quickly reduces to the case where $f \colon X \hookrightarrow Y$ 
  is a quasi-smooth embedding and one has a smoothing
  \[
    \xymatrix{
      & P \ar[d]^{p} \\
      X \ar@{^{(}->}[ur]^{i} \ar@{^{(}->}[r]_{f} & Y.
    }
  \]
  We then have the Cartesian and homotopy Cartesian diagram
  \[
    \xymatrix{
      P' \ar@{^{(}->}[r]^{f'} \ar@<0.5ex>[d]^{p'} & P \ar[d]^{p} \\
      X \ar@<0.5ex>[u]^{s} \ar@{^{(}->}[r]_{f} & Y.
    }
  \]
  with $f' \circ s = i$. Since $p' \circ s = \id$ and $p'$ is smooth, we 
  obtain that $L_s$ is of Tor-dimension $\leq 1$. The morphism $s$ is an 
  embedding since $f^\prime$ and $i = f^\prime \circ s$ are embeddings 
  (in particular, $f^\prime$ is separated), and thus $s$ is a 
  quasi-smooth embedding. By Proposition \ref{prop:FuncQSE}, we have
  \[
    i^*p^* = s^*f^{'*}p^*.
  \]
  Using Lemma \ref{lem:QSPushPull}(a), the right hand side is equal to 
  $s^*p^{'*}f^*$. Since $s^*p^{'*}= \id$, the claim follows.
\end{proof}
  
This allows us to define orientations for quasi-smooth morphisms in 
$\Omega_*$.

\begin{defn}
  \label{defn:PbQs}
  Let $f \colon X \to Y$ be a quasi-smooth morphism in $\Cat{dQPr}_k$, 
  and let $X \overset{i}{\hookrightarrow} P \overset{p}{\to} Y$ be a 
  smoothing of $f$ with $p$ of relative dimension $n$ and $i$ of 
  virtual co-dimension $m$. Let $d=n-m$. We then define 
  \emph{quasi-smooth pull-back} as the composition
  \[
    f^* \colon  \Omega_*(Y) \overset{p^*}{\longrightarrow} 
    \Omega_{*+n}(P) \overset{i^*}{\longrightarrow} \Omega_{*+d}(X).
  \]
  Here $i^*$ is the quasi-smooth pull-back for quasi-smooth embeddings 
  of Definition \ref{defn:QsPbOm}.
\end{defn}

\begin{thm}
  Let $k$ be a field of characteristic zero. Then $\Omega_{*}$ is an 
  oriented Borel--Moore functor on $\Cat{dQPr}_k$ of geometric type with 
  quasi-smooth pullbacks.
\end{thm}
\begin{proof}
  It remains to verify the axioms of Definition \ref{defn:BMFqs}.  To 
  prove functoriality and thus (QS2), choose smoothings as in 
  \cite[Remark 5.1.2]{lm} such that we obtain a ladder of smoothings
  \[
    \xymatrix{
      X \ar[r] \ar[dr] & P_1 \ar[r] \ar[d] & P \ar[d] \\
      & Y \ar[r] \ar[dr]& P_2 \ar[d] \\
      & & Z.
    }
  \]
  with the square Cartesian and homotopy Cartesian. To prove axiom 
  (QS3), factor $f$ as $p \circ i$. The statement then immediately 
  follows from the individual statements for $p$ and $i$.  For $p$ the 
  statement is clear since $p$ is smooth, for $i$ this is Lemma 
  \ref{lem:QSPushPull}(b). The proof of (QS4) is straight forward from 
  the definitions. Since the axiom (QS5) only applies to smooth and thus 
  underived schemes, we can deduce it from the corresponding statement 
  in \cite[Proposition 7.2.2]{lm}.
\end{proof}

Since $d\Omega_*$ is the universal Borel--Moore functor with orientations 
for quasi-smooth morphisms we obtain the following Corollaries.

\begin{cor}
  \label{cor:dOmToOm}
  Let $k$ be a field of characteristic zero.
  We then obtain a classifying morphism
  \[
    \vartheta_{\Omega} \colon d\Omega_{*} \to \Omega_{*}.
  \]
\end{cor}
\begin{proof}
  This immediately follows from the universal property of $d\Omega_*$ as 
  the universal Borel-Moore homology theory on $\Cat{dQPr}_k$ with 
  quasi-smooth pullbacks. Since $\Omega_*$ is also such a theory, the 
  existence of the classifying morphism $\vartheta_{\Omega}$ follows.
\end{proof}

\begin{cor}
  \label{cor:EasyHalf}
  For all $X \in \Cat{dQPr}_k$, the natural transformation
  \[
    \vartheta_{d\Omega} \colon \Omega_*(X) \longrightarrow d\Omega_*(X)
  \]
  is injective.
\end{cor}
\begin{proof}
  We want to show that $\vartheta_{\Omega}$ is a left inverse. To 
  prevent confusion, in the following we will decorate pull-backs and 
  push-forwards with the homology theory they are taken in. Let $f 
  \colon Y \to X$ be a bordism cycle.  Since $\vartheta_{\Omega}$  
  commutes with proper push-forwards and smooth pull-backs we have
  \begin{align*}
    \vartheta_{\Omega} \circ \vartheta_{d\Omega}([Y \to X])&= 
    \vartheta_{\Omega}(f^{d\Omega}_*\pi_{Y,d\Omega}^*[1]) \\
    &=f^{\Omega}_*\pi_{Y,\Omega}^* \vartheta_{\Omega}[1]\\ &= [Y \to X].
  \end{align*}
\end{proof}

\section{Spivak's Theorem}
\label{sect:Spivak}

We now have two Borel--Moore functors of geometric type on 
$\Cat{dQPr}_k$ at our disposal, derived algebraic bordism $d\Omega_*$ 
and algebraic bordism $\Omega_*$. Both are equipped with orientations 
for quasi-smooth morphisms. We have also constructed natural 
transformations $\vartheta_{d\Omega}\colon \Omega_* \to d\Omega_*$ and 
$\vartheta_{\Omega}\colon d\Omega_* \to \Omega_*$ between these 
theories. Here $\vartheta_{d\Omega}$ is induced by the universal 
property of $\Omega_*$ as the universal Borel-Moore homology theory of 
geometric type on $\Cat{dQPr}_k$, and $\vartheta_{\Omega}$ is induced by 
the universal property of $d\Omega$ as the universal Borel-Moore 
homology theory on $\Cat{dQPr}_k$ with quasi-smooth pull-backs.

The goal of this section is to compare how the orientations for 
quasi-smooth morphisms of these two theories interact.  In the following 
we will prove a Grothendieck--Riemann--Roch type result stating that 
$\vartheta_{d\Omega}$ in fact commutes with these orientations. As a 
direct consequence of this result we obtain an algebraic version of 
Spivak's theorem for derived cobordism in the differentiable setting.

The strategy for proving our Grothendieck--Riemann--Roch type result is 
the same as in the classical case treated in \cite{verdier} or 
\cite{fulton}. We first check the compatibility of the orientations in 
specific settings, and later reduce the general case to the known 
setting.

A word on notation: In the following we will often encounter formulas 
involving pull-backs and push-forwards in different homology theories.  
Except in special cases, we have chosen not to decorate these operations 
with the homology theories they are taken in. We hope this does not lead 
to confusion.

\subsection{Preliminaries on $d\Omega_*$}

For certain types of bordism classes, we want to reduce pull-back along 
a quasi-smooth embedding in $d\Omega_*$ to intersecting with a effective 
Cartier divisor. The key ingredient in this step is again deformation to 
the normal cone.

Recall from Corollary \ref{cor:EasyHalf} that $\vartheta_{d\Omega}$ is 
injective. We let $d\Omega_*^{\cl}$ denote the image of 
$\vartheta_{d\Omega}$. For any $X \in \Cat{dQPr}_k$, classes in 
$d\Omega_*^{\cl}$ are thus given by proper morphisms $[Y \to X]$ where 
$Y$ is assumed to be smooth.

\begin{lem}
  \label{lem:FactorOverN}
  Let $f \colon X \hookrightarrow Y$ be a quasi-smooth embedding, and 
  let $g \colon Y' \to Y$ be a morphism of derived schemes such that $g$ 
  factors over the normal bundle $N_X Y$. Then we have an equivalence of 
  derived schemes
  \[
    Y' \times_Y^h X \simeq Y' \times_{N_X Y}^h X.
  \]
\end{lem}
\begin{proof}
  This follows immediately from the criterion that a morphism of derived 
  schemes is an equivalence if it induces an isomorphism on the 
  truncation $t_0$ and an equivalence of the cotangent complexes.
\end{proof}

\begin{prop}
  \label{prop:RedToCar}
  Let $f \colon X \to Y$ be a quasi-smooth embedding, and let $u \in 
  d\Omega_*^{\cl}(Y)$. Then there exists a class $u' \in 
  d\Omega_*^{\cl}(N)$ such that
  \[
    f^* (u) = s^* (u').
  \]
\end{prop}
\begin{proof}
  We let $M$ be the truncation $t_0(M_X Y)$ of the deformation to the 
  normal cone space of $X$ in $Y$, and let $\pr \colon Y \times \IA^1 
  \to Y$ be the projection on the first factor. We then obtain a class 
  $\pr^* u \in d\Omega_{*+1} (Y \times \IA^1)$. Identify $Y \times 
  \IA^1$ with an open subset of $M$.  Since $\pr^* u$ is in 
  $d\Omega^{\cl}_{*+1}(Y \times \IA^1)$, we can extend it to a class 
  $\tilde{u}$ in $d\Omega^{\cl}_{*+1}(M)$. Since the fiber over zero of 
  $M$ is the virtual normal bundle $N_{X}Y$, the class $\tilde{u}$ 
  provides a bordism between $u \in d\Omega_*^{\cl}(X)$ and a class $u' 
  \in d\Omega_*(N)$.  The remaining claim follows from Lemma 
  \ref{lem:FactorOverN}.
\end{proof}

\begin{rem}
  \label{rem:AfterLiftDOm}
  Once we haven chosen a lifting $\tilde{u}$, we obtain the following 
  explicit formula for quasi-smooth pullback:
  \[
    f^*(u) = s^* \circ i^*(\tilde{u})
  \]
  This reduces all questions about quasi-smooth pull-backs to 
  intersection with a Cartier divisor followed by pull-back along a 
  zero-section.
\end{rem}

%

We need to compare our Euler classes with more traditionally defined 
Chern classes in cases of overlap. The injectivity of 
$\vartheta_{d\Omega}$ allows us to define Chern class operators on 
$d\Omega_*^{\cl}$ as follows. With $X \in \Cat{QPr}_k$, let $E \to X$  
be a vector bundle of rank $e+1$ with associated projective bundle $q 
\colon \IP(E) \to X$.  We denote the universal line bundle as $\sO(1)$ 
and let $\xi$ be its first Chern class operator on $d\Omega_*(\IP(E))$.  
Composing with pull-back along $q$,  we have operators
\[
  \phi_j := \xi^j \circ q^* \colon d\Omega_*(X) \longrightarrow 
  d\Omega_{*+e-j}(X).
\]
Let
\[
  \Phi \colon \bigoplus_{j=0}^{n}  d\Omega_{*-e+j} (X)\longrightarrow 
  d\Omega_{*}(\IP(E))
\]
be the sum over the operators $\phi_j$. 

\begin{lem}
  \label{lem:PBT}
  The morphism
  \[
    \Phi \colon \bigoplus_{j=0}^{n} d\Omega^{\cl}_{*-e+j}(X) 
    \longrightarrow d\Omega^{\cl}_{*}(\IP(E))
  \]
  is an isomorphism.
\end{lem}
\begin{proof}
  Since $d\Omega_*^{\cl}$ is precisely the image of 
  $\vartheta_{d\Omega}$, this follows from the fact that 
  $\vartheta_{d\Omega}$ commutes with smooth pull-back and first Chern 
  class operators, and that the Projective Bundle Theorem holds in 
  $\Omega_*$.
\end{proof}
The following corollary reduces many 
questions about vector bundles to sums of line bundles via the 
splitting principle.
\begin{cor}
  \label{cor:qInj}
  The morphism
  \[
    q^* \colon d\Omega_*^{\cl}(X) \longrightarrow 
    d\Omega_{*+e}^{\cl}(\IP(E))
  \]
  is injective.
\end{cor}

Using the method of \cite[Section 4.1.7]{lm}, as an immediate 
consequence of the Projective Bundle Theorem we obtain the existence of 
uniquely defined operators
\[
  \tilde{c}_i(E) \colon d\Omega_*^{\cl}(X) \longrightarrow 
  d\Omega^{\cl}_{*-i}(X)
\]
satisfying the relations
\[
  \sum_{i=0}^{n} (-1)^i c_1(\sO(1))^{n-i} \circ q^* \circ 
  \tilde{c}_i(E)=0.
\]
Note that by construction  $c_1=\tilde{c}_1$. These Chern classes 
satisfy all the expected properties, e.g., the Whitney product formula. In 
particular, if $E = \oplus_{i=0}^{e}L_i$ is a direct sum of line 
bundles, then the $j$-th Chern class of $E$ is the $j$-th symmetric 
polynomial in the $c_1(L_i)$.

A further immediate consequence of Lemma \ref{lem:PBT} is an explicit 
formula for the pull-back along a zero-section of a vector bundle.
\begin{cor}
  \label{cor:PBZeroSec}
  Let $X \in \Cat{QPr}_k$, let $p\colon E \to X$ be a rank $r$ vector 
  bundle, and let $q \colon \IP(E \oplus 1) \to X$ be the corresponding 
  projective bundle with universal quotient bundle $\xi$. Let $i \colon 
  E \hookrightarrow \IP(E \oplus 1)$ denote the canonical open 
  embedding.  Then for any $u \in d\Omega_*^{\cl}(\IP(E \oplus 1))$, we 
  have
  \[
    (p^*)^{-1} (i^*u) = q_*(c_r(\xi) \cap u).
  \]
\end{cor}

We show compatibility with the definition of the top Chern class given 
in Definition \ref{defn:c1}.

\begin{lem}
  \label{lem:SelfInt}
  Let $X \in \Cat{QPr}_k$, and let $E \to X$ be a vector bundle of rank 
  $e$. Then
  \[
    c_e(E) \cap u = \tilde{c}_e(E) \cap u
  \]
  for all $u \in d\Omega_*^{\cl}(X)$.
\end{lem}
\begin{proof}
  Let $q \colon \IP(E) \to X$ be the projection. By the injectivity of 
  $q^*$, we can apply the splitting principle and assume that $E$ is a 
  direct sum of line bundles $\oplus _{i=1}^e L_i$. Then 
  $\tilde{c}_e(E)=c_1(L_1) \circ \dots \circ c_1(L_e)$ by the Whitney 
  product formula.

  We prove the claim by induction on rank. Let $u=[Y \to X]$. For 
  $e=1$ the claim holds by construction of the Chern classes. 
  Unravelling the definitions, the left hand side is given by the 
  homotopy fiber product
  \begin{equation}
    \label{eq:bla}
    \xymatrix{
      Y' \ar[r] \ar[dd] & Y \ar[d] \\
      & X \ar[d]_{s_0} \\
      X \ar[r]_(.4){s_0} & L_1 \oplus \dots \oplus L_e
    }
  \end{equation}
  By induction, the right hand side is given by the homotopy fiber 
  product
  \[
    \xymatrix{
      Y^{\prime \prime} \ar[r] \ar[ddd] & Y' \ar[r] \ar[dd] & Y \ar[d] 
      \\
      &&X \ar[d]^{s_0} \\
      & X \ar[d]_{s_0} \ar[r]^(.35){s_0} & L_1 \oplus \dots L_{e-1} \\
      X \ar[r]_{s_0} & L_e
    }
  \]
  Using the universal property of the homotopy fiber product 
  \eqref{eq:bla}, we obtain a morphism $Y^{\prime \prime} \to Y'$. Since 
  we have always used the zero-sections, this is an isomorphism on the 
  underlying classical schemes. By a straight-forward calculation, it 
  also induces an equivalence of the cotangent complexes. It thus is an 
  equivalence.
\end{proof}

\subsection{A Grothendieck--Riemann--Roch result}
We next show that $\vartheta_{d\Omega}$ commutes with Chern class 
operators and intersection with effective Cartier divisors.

\begin{lem}
  \label{lem:CrCommutes}
  Let $X \in \Cat{QPr}_k$ and let $E \to X$ be a vector bundle. Then the 
  following diagram commutes:
   \[
    \xymatrix{
      \Omega_*(X) \ar[r]^{\vartheta_{d\Omega}} \ar[d]_{c_p(E) \cap  
        _{-}}& d\Omega_*(X) \ar[d]^{c_p(E) \cap _{-}} \\
      \Omega_{*-r}(X) \ar[r]^{\vartheta_{d\Omega}} & d\Omega_{*-r} (X)
    }
  \]
\end{lem}
\begin{proof}
  Let $q \colon \IP(E) \to X$ be the projection. Since $q^*$ is 
  injective both for $\Omega_*$ and $d\Omega_*^{\cl}$, we can apply the 
  splitting principle and  assume  $E$ is a direct sum of line bundles 
  $L_i$.  Then in both $\Omega_*$ and $d\Omega_*$, the operation $c_p 
  (E)$ is the $p$-th symmetric polynomial in the $c_1(L_i)$.  Since 
  $\vartheta_{d\Omega}$ commutes with first Chern classes, the claim 
  follows.
\end{proof}

For the proof of the next lemma we have to recall one of the key 
technical results in \cite{lm}. Levine and Morel prove that for any 
finite type scheme $X$ over $k$ and any pseudo-divisor $D$ on $X$ there 
exists an isomorphism between $\Omega_*(X)$ and a group $\Omega_*(X)_D$, 
whose generators are classes $[Y \overset{f}{\to} X]$ such that either 
$f(Y)$ is contained in $D$ or $f^*D$ is a strict normal crossing 
divisor. We call this result \emph{Levine's moving lemma} \cite[Thm.  
6.4.12]{lm}.

\begin{lem}
  \label{lem:IDcommutes}
  Let $i_D \colon D \hookrightarrow X$ be an effective Cartier divisor 
  on a scheme $X$. Then the following diagram commutes:
  \begin{equation}
    \label{eq:IDcomm}
    \xymatrix{
      \Omega_*(X) \ar[r]^{\vartheta_{d\Omega}} \ar[d]_{i_{D}^*}& 
      d\Omega_*(X) \ar[d]^{i_D^*} \\
      \Omega_{*-1}(D) \ar[r]^{\vartheta_{d\Omega}} & d\Omega_{*-1} (D)
    }
  \end{equation}
\end{lem}
\begin{proof}
  By Levine's moving lemma, it suffices to treat the cases of bordism 
  cycles $f \colon Y \to X$ such that $f$ factors through $D$ or 
 the fiber product of $f$ and $i_D$ is a strict normal crossing divisor of $Y$.

  We first assume that $f$ factors through $D$. Let $f^D \colon Y \to 
  D$ be the induced morphism, i.e., $f = i_D \circ f^D$.  By definition    
  \begin{align*}
 \vartheta_{d\Omega} \circ i^\ast_D = \vartheta_{d\Omega}(f^D_*(c_1(f^*\sO_X(D)) \cap 
  1_Y))
\end{align*}
On the other hand, by Lemma \ref{lem:FactorOverN}
\begin{displaymath}
  i_D^\ast \circ \vartheta_{d\Omega} = c_1(i_D^*\sO_X(D)) 
 \cap [Y \overset{f^D}{\to}  D]   
\end{displaymath}
 The  claim  follows by applying the projection formula (A3) of the axioms of 
 Borel--Moore functor and observing that $\vartheta_{d\Omega}$ commutes 
 with proper push-forward and first Chern classes.

  Let us now assume that $D^\prime := D \times_X Y$ is a strict normal 
  crossing divisor of $Y$.  This implies that we have a 
  Tor-independent\footnote{Recall that this means $\Tor_i^{\sO_X}(\sO_Y, 
    \sO_D)=0$ for $i>0$.} diagram:
  \[
    \xymatrix{
      D^\prime \ar[r]^{i_{D^\prime}} \ar[d]_{f'} & Y \ar[d]^{f} \\
      D \ar[r]_{i_D} & X
    }
  \]
  In this case, 
  \begin{align*}
 \vartheta_{d\Omega} \circ i^\ast_D = \vartheta_{d\Omega}(  f'_*(\zeta_{D*}( [D^\prime \to |D^\prime|])))
\end{align*}
Recall that the class $[D^\prime \to |D^\prime|]$ is not given by a morphism, but instead by equation (\ref{eq:FundSNC}).   Since the above diagram is Tor-independent, it is a homotopy fiber diagram and thus 
  \begin{align*}
    i_D^* \circ \vartheta_{d\Omega}([Y \to X]) &= i_D^*f_* (1_Y) \\&= f'_* i_{D^\prime}^* 1_Y \\
    &= f'_* 1_{D^\prime}.
  \end{align*}
  The claim then follows by applying the relation $\sR_*^{\SNC}$ and 
  observing that $\vartheta_{d\Omega}$ commutes with proper 
  push-forwards, first Chern classes and is compatible with the formal 
  group law.
\end{proof}

\begin{thm}
  \label{thm:commutes}
  Let $f \colon X \hookrightarrow Y$ be a quasi-smooth embedding of 
  virtual codimension $d$ in $\Cat{dQPr}_k$ with $Y$ a scheme.  Then the 
  following diagram commutes:
  \[
    \xymatrix{
      \Omega_*(Y) \ar[r]^{\vartheta_{d\Omega}} \ar[d]_{f^*_{\Omega}} & 
      d\Omega_*(Y)\ar[d]_{f^*_{d\Omega}} \\
      \Omega_{*-d}(X) \ar[r]^{\vartheta_{d\Omega}} & d\Omega_{*-d}(X).
    }
  \]
\end{thm}
\begin{proof}
  Let $M$ be the truncation of the deformation to the normal cone space 
  $M_X^{\circ}Y$. By Remark \ref{rem:AfterLift} and Remark 
  \ref{rem:AfterLiftDOm}, after we have chosen a lifting $\tilde{u} \in 
  \Omega_{*+1}(M)$ of $\pr^*u \in \Omega_{*+1}(Y \times \IA^1)$, 
  pull-back along $f$ is given by $s^* \circ i^*$ both in $\Omega_*$ and 
  $d\Omega_*$. Here $i \colon N \to M$ is the inclusion of the virtual 
  normal bundle $N$ as Cartier divisor on $M$, and $s \colon t_0(X) \to 
  N$ is the zero-section. Note that by the localization theorem for 
  $\Omega_*$ a suitable lifting $\tilde{u}$ always exists.

  The case of intersection with a Cartier divisor was treated in Lemma 
  \ref{lem:IDcommutes}. It remains to prove that $\vartheta_{d\Omega}$ 
  commutes with intersection with the zero-section. So let $X \in 
  \Cat{QPr}_k$, let $p \colon E \to X$ be a vector bundle, let $q \colon 
  \IP(E \oplus 1)$ be the associated projective bundle with open 
  embedding $i \colon E \to \IP(E \oplus 1)$ and universal quotient 
  bundle $\xi$, and let $u \in \Omega_*(E)$.  By the localization 
  theorem for $\Omega_*$, there exists a class $\tilde{u} \in 
  \Omega_*(\IP (E \oplus 1))$ such that $i^* (\tilde{u}) = u$. Since the 
  Projective Bundle Theorem holds for $\Omega_*$, going around the left 
  hand side of our diagram is given by
  \[
    \vartheta_{d\Omega} \circ s^* (u) = \vartheta_{d\Omega} \circ q_* ( 
    c_d (\xi) \cap \tilde{u})
  \]
  On the other hand, since $\vartheta_{d\Omega}(u)$ by definition lies 
  in $d\Omega_*^{\cl}$, we can apply Corollary \ref{cor:PBZeroSec} to 
  compute going around the right hand side:
  \[
    s^* \circ \vartheta_{d\Omega}(u) = q_* (c_d (\xi) \cap 
    \vartheta_{d\Omega}(\tilde{u})).
  \]
  Since $\vartheta_{d\Omega}$ commutes with Chern classes by Lemma 
  \ref{lem:CrCommutes} and with proper push-forward, the claim follows.
\end{proof}

We can now use the previous theorem to obtain a 
Grothendieck--Riemann--Roch type result for the natural transformation 
$\vartheta_{d\Omega}$.

\begin{cor}[Grothendieck--Riemann--Roch for $\vartheta_{d\Omega}$]
  \label{cor:GRR}
  Let $f \colon X \to Y$ be a quasi-smooth morphism of relative virtual 
  dimension $d$ in $\Cat{dQPr}_k$. Then the following 
  diagram commutes:
  \[
    \xymatrix{
      \Omega_*(Y) \ar[r]^{\vartheta_{d\Omega}} \ar[d]_{f^*_{\Omega}} & 
      d\Omega_*(Y)\ar[d]_{f^*_{d\Omega}} \\
      \Omega_{*+d}(X) \ar[r]^{\vartheta_{d\Omega}} & d\Omega_{*+d}(X).
    }
  \]
\end{cor}
\begin{proof}
  Let $u = [V \overset{g}{\to} Y] \in \Omega_*(Y)$ be a bordism cycle.  
  By definition, $u = g_\ast 1_V$ where $g\colon V \to Y$ is proper and 
  $V$ is smooth. By factoring $f$ as a quasi-smooth embedding and a 
  smooth morphism, we can assume $f$ is an embedding (the case of $f$ 
  smooth is clear). We can now form the homotopy fiber product
  \[
    \xymatrix{
      W \ar[r]^{f'} \ar[d]_{g'} &V \ar[d]^{g}\\
      X \ar[r]_{f}&Y
    }
  \]
  Note that $f' \colon W \to V$ is a quasi-smooth embedding with $V$ 
  smooth. Thus Theorem \ref{thm:commutes} is applicable to $f'$.  By 
  Property (QS3) $f^* u = f^\ast g_\ast 1_V = g^\prime_\ast f^{\prime 
    \ast} 1_V$.  Applying $\vartheta_{d\Omega}$ to this equation 
  \begin{align*}
    \vartheta_{d\Omega}(f^\ast u) & = \vartheta_{d\Omega} (g^\prime_\ast 
    f^{\prime \ast} 1_V) \\  & = g_\ast^\prime 
    \vartheta_{d\Omega}(f^{\prime \ast} 1_V) \\  & = g_\ast^\prime 
    f^{\prime \ast} \vartheta_{d\Omega} ( 1_V)\\
& = f^\ast g_\ast \vartheta_{d\Omega}(1_V)\\
& = f^\ast \vartheta_{d\Omega}(g_\ast 1_V)\\
& = f^\ast \vartheta_{d\Omega}(u)
  \end{align*}
the result immediately follows.
\end{proof}

We are now able to deduce an algebraic version of Spivak's theorem from 
the Grothendieck--Riemann--Roch Theorem for $\vartheta_{d\Omega}$.

\begin{thm}[Algebraic Spivak Theorem]
  \label{thm:dOmEqOm}
  For all $X \in \Cat{dQPr}_k$ the morphism
  \[
    \vartheta_{d\Omega} \colon \Omega_{*} (X)
    \longrightarrow d\Omega_{*}(X)
  \]
  is an isomorphism.
\end{thm}
\begin{proof}
  By Lemma \ref{cor:EasyHalf}, $\vartheta_{d\Omega}$ is a right inverse 
  to $\vartheta_{\Omega}$. We have to show that it is also a left 
  inverse.  Let $f \colon Y \to X$ be a derived bordism cycle. By 
  Corollary \ref{cor:GRR}, $\vartheta_{d\Omega}$ commutes with pull-backs 
  along quasi-smooth morphisms. We then have
  \begin{align*}
    \vartheta_{d\Omega} \circ \vartheta_{\Omega}([Y \to X]) &= 
    \vartheta_{d\Omega}(f^\Omega_*\pi_{Y,\Omega}^*[1])\\
    &=f^{d\Omega}_*\vartheta_{d\Omega}(\pi_{Y,\Omega}^*[1])\\
    &=f^{d\Omega}_*\pi_{Y,d\Omega}^*\vartheta_{d\Omega}( [1])\\
    &= [Y \to X].
  \end{align*}
\end{proof}

Since there exist non-trivial derived schemes of negative virtual 
dimension we by no means know a priori that $d\Omega_n(X)=0$ for $n < 
0$. This is immediately implied by the previous theorem.

\begin{cor}
  Let $n < 0$. For all $X \in \Cat{dQPr}_k$ we have
  \[
    d\Omega_n(X) = 0
  \]
\end{cor}

Theorem \ref{thm:dOmEqOm} provides a geometric explanation why virtual 
fundamental classes exist for quasi-smooth derived schemes. It implies 
that every quasi-smooth projective derived scheme bords to a smooth 
projective scheme. We make this precise in the following corollary.  We 
write $[Y]$ for the class $[Y \to \pt]$ in $d\Omega_*(k)$.

\begin{cor}
  \label{cor:AllBord}
  For every projective quasi-smooth derived scheme $X$ of virtual 
  dimension $n$ there exists a smooth projective scheme $Y$ of dimension 
  $n$ such that
  \[
    [Y]=[X] \in d\Omega_n(k).
  \]
\end{cor}
\begin{proof}
  By Theorem \ref{thm:dOmEqOm} the morphism $\vartheta_{d\Omega} \colon 
  \Omega_n(k) \to d\Omega_n(k)$ is an isomorphism. Since the image of 
  $\vartheta_{d\Omega}$ consist of smooth projective schemes the claim 
  follows.
\end{proof}

A further consequence of the comparison theorem is that the Projective 
Bundle Theorem, the Extended Homotopy relation as well as the cellular 
decomposition relation hold for $d\Omega_*$.
\begin{cor}
  $d\Omega_*$ is an oriented Borel-Moore homology theory with 
  quasi-smooth pull-backs as in Definition \ref{defn:OBMHTqs}.
\end{cor}

\begin{rem}
  In differential geometry, the corresponding result of Theorem  
  \ref{thm:dOmEqOm} admits a direct proof \cite{joyce, spivak}. There 
  every derived manifold in either the sense of Spivak or Joyce admits a 
  global presentation as zero-set of a section of a vector bundle. Such 
  a derived manifold can be made bordant to a manifold by arranging the 
  section to be transverse.
\end{rem}

\section{Derived Algebraic Cobordism}

In this final section we want to study the cohomology theory $d\Omega^*$ 
on smooth projectve schemes $\Cat{Sm}_k$ associated to the Borel-Moore 
homology theory $d\Omega_*$.

For $X \in \Cat{Sm}_k$ of pure dimension $d$, we set $d\Omega^n(X) := 
d\Omega_{d-n}(X)$. The comology theory $d\Omega^*$ is then equipped with 
an external product. Since for a smooth scheme $X$ the diagonal 
embedding $\delta \colon X \to X \times X$ is a local complete 
intersection morphism, we can define an intersection product via
\[
  d\Omega^*(X) \otimes d\Omega^*(X) \longrightarrow d\Omega^*(X \times 
  X) \overset{\delta^*}{\longrightarrow} d\Omega^*(X).
\]
Here the first arrow is the external product. Given two cobordism 
classes $[Y \to X]$ and $[Z \to X]$, we will denote the intersection 
product by $[Y] \cdot [Z]$. The intersection product turns 
$d\Omega^*(X)$ into a commutatve graded ring with unit $[X 
\overset{\id}{\to} X]$. Unravelling the definitions, the intersection 
product is represented by the homotopy fibre product
\begin{equation}
  \label{eq:IntProd}
  \xymatrix{
    (Y \times Z) \times_{X \times X}^h X \ar[r] \ar[d] & Y \times Z 
    \ar[d] \\
    X \ar[r] & X \times X
  }
\end{equation}

For the proof ot next theorem, recall that for two subsets $U$ and $V$ 
of a set $X$, and denoting by $\Delta$ the diagonal of $X$ in $X \times 
X$, we have the set-theoretic identity $(U \times V) \cap \Delta = U 
\cap V$. In homological algebra, given a $k$-algebra $A$ and two 
$A$-modules $M$ and $N$, this translates to the identity $(M \otimes _k 
N) \otimes_{A \otimes A} A \simeq M \otimes _A N$. We refer to this as 
\emph{reduction to the diagonal}. Using this, We can then prove the 
following result on the intersection product:

\begin{thm}
  \label{thm:IntProd}
  Let $X \in \Cat{Sm}_k$. Then the intersection product is represented 
  by the homotopy fibre product:
  \[
    [Y] \cdot [Z] = [Y \times^h_X Z] \in d\Omega^*(X).
  \]
\end{thm}
\begin{proof}
  Since the base $k$ is a field, this follows from reduction to the 
  diagonal as in Serre \cite[V.B.1]{serre}.
\end{proof}

In particular, since the natural transformation $\vartheta_{\Omega}$ 
commutes with the intersection product, we obtain the formula
\[
  [Y] \cdot [Z] = \vartheta_{\Omega} ([Y \times^h_X Z])
\]
in algebraic cobordism $\Omega^*(X)$ for algebraic cobordism classes $[Y 
\to X]$ and $[Z \to X]$.

\bibliographystyle{amsplain}
\bibliography{biblio}
\end{document}